\documentclass[12pt]{amsart}
\usepackage{amsmath,amssymb,amsthm}

\newtheorem{theorem}{Theorem}
\newtheorem{definition}[theorem]{Definition}
\newtheorem{lemma}[theorem]{Lemma}

\newtheorem{proposition-definition}[theorem]{Proposition-Definition}
\newtheorem{corollary}[theorem]{Corollary}

\theoremstyle{definition}
\newtheorem{example}[theorem]{Example}

\theoremstyle{remark}
\newtheorem*{remark}{Remark}

\begin{document}
\author[Hamblen, Jones, and Madhu]{Spencer Hamblen, Rafe Jones, and Kalyani Madhu}
\title{The density of primes in orbits of $z^d+c$}
\thanks{The second author's research was partially supported by NSF grant DMS-0852826.}

\maketitle

\begin{abstract}
Given a polynomial $f(z) = z^d + c$ over a global field $K$ and $a_0 \in K$, we study the density of prime ideals of $K$ dividing at least one element of the orbit of $a_0$ under $f$.  The density of such sets for linear polynomials has attracted much study, and the second author has examined several families of quadratic polynomials, but little is known in the higher-degree case. We show that for many choices of $d$ and $c$ this density is zero for all $a_0$, assuming $K$ contains a primitive $d$th root of unity. The proof relies on several new results, including some ensuring the number of irreducible factors of the $n$th iterate of $f$ remains bounded as $n$ grows, and others on the ramification above certain primes in iterated extensions. Together these allow for nearly complete information when $K$ is a global function field or when $K={\mathbb{Q}}(\zeta_d)$.
\end{abstract}

\maketitle

\section{Introduction}
Let $K$ be a field, and let $f(z)=z^d+c\in K[z]$.  For $n \geq 1$, denote by $f^n(z)$ the $n$th iterate of $f$, and set $f^0(z) = z$. By the {\em orbit} of $a_0 \in K$ under $f$, we mean the set 
\[O_f(a_0) = \{f^n(a_0) : n \geq 0\}.\]

When $K$ is a global field, we denote by ${\mathcal{O}}_K$ the usual ring of integers of $K$ (in the number field case) or the integral closure in $K$ of ${\mathbb{F}}_q[t]$ (in the function field case). We say that a prime ideal ${\mathfrak{q}} \in {\mathcal{O}}_K$ \textit{divides} $O_f(a_0)$ if there exists at least one $n \geq 0$ with $f^n(a_0) \neq 0$ and  $v_{\mathfrak{q}}( f^n(a_0))>0$. Our purpose in this article is to study the set of prime ideals
\[P_{f}(a_0) = \{{\mathfrak{q}} \subset {\mathcal{O}}_K : \text{${\mathfrak{q}}$ divides $O_f(a_0)$}\},\]
and in particular to show that in many circumstances it is sparse within the set of all prime ideals of ${\mathcal{O}}_K$. This problem has applications to the dynamical Mordell-Lang conjecture \cite{ML} and to questions about the size of the set of hyperbolic maps in $p$-adic multibrot sets \cite{galmart}. It is also studied in \cite{periodsmoduloprimes}, where it is shown under much more general hypotheses that the density of $P_f(a_0)$ is less than one; here our goal is to show that $P_f(a_0)$ has density zero when $K$ contains a primitive $d$th root of unity. The set $O_f(a_0)$ may also be considered as a non-linear recurrence sequence, and in this guise the question of the density of $P_{f}(a_0)$ has been much studied (see \cite{quaddiv} for a brief overview, and \cite{ballot, recseq} for more comprehensive studies).   The family $f(z) = z^d + c$ is a natural candidate for study in this regard, since many of the arithmetic properties of the orbits of a polynomial depend on the orbits of its critical points, and this family has only one critical point. That this critical point is zero also plays a key role, since it ensures that the critical orbit has a property we call rigid divisibility; see Section \ref{irred}. See recent work in \cite{ingramcanon} and \cite{krieger} for other arithmetic dynamical properties of this family.

Denote by $D(S)$ the Dirichlet density of a set $S$ of primes of $K$, i.e.
\[D(S) = \lim_{s \rightarrow 1^{+}} \frac{\sum_{{\mathfrak{q}} \in S} N({\mathfrak{q}})^{-s}} {\sum_{{\mathfrak{q}}}N({\mathfrak{q}})^{-s}},\]
where $N({\mathfrak{q}}) = \#({\mathcal{O}}_K/{\mathfrak{q}}{\mathcal{O}}_K)$, and the sum in the denominator runs over all primes of $K$. In the number field case, we may replace this with the more intuitive notion of natural density:
\[D(S) = \limsup_{x \to \infty} \frac{\#\{q \in S : N({\mathfrak{q}}) \leq x\}}{\#\{{\mathfrak{q}} : N({\mathfrak{q}}) \leq x\}},\]
We remark that the set $P_f(a_0)$ is infinite unless $O_f(a_0)$ is finite or $f(z) = cz^d$, as can be shown by trivial modifications to \cite[Theorem 6.1]{quaddiv}.

To state our main result, we take the set $M_K$ of places of $K$ to be a complete set of inequivalent absolute values on $K$, each extending one of the standard absolute values on ${\mathbb{Q}}$ or ${\mathbb{F}_q}(t)$.  (By a standard absolute value on ${\mathbb{F}_q}(t)$, we mean $|x| = q^{-v(x)}$, where $v$ is the valuation corresponding to a prime  of ${\mathbb{F}_q}[t]$ or the degree map.) Each non-archimedean $v \in M_K$ has an associated residue field $\{|x|_v \leq 1\} / \{|x|_v < 1\}$, whose characteristic is the \textit{residue characteristic} of $v$.

\begin{theorem} \label{main1}
Let $K$ be a global field containing a primitive $d$th root of unity, and let $f(z)=z^d+c$. Suppose $c \in K$, $O_f(0)$ is infinite, and one of the following holds:
\begin{itemize}
	\item[(1)] There exists a non-archimedean $v \in M_K$ such that $|c|_v < 1$ and the residue characteristic of $v$ is prime to $d$; or
	\item[(2)] $d$ is prime and for some $j \geq 0$, $f^j(z) = g_1(z) \ldots g_t(z)$ with each $g_i$ irreducible and none of $\pm g_i(f(0)), g_i(f^2(0)), g_i(f^3(0)), \ldots$ is a $d$th power in $K$.
\end{itemize}
Then $D(P_f(a_0)) = 0$ for any $a_0 \in K$.
\end{theorem}

Condition (2) is often applied when $j = 0$, in which case it holds when none of $\pm f(0), f^2(0), f^3(0), \ldots$ is a $d$th power in $K$. The $\pm$ attached to $g_i(f(0))$ is in fact $-1$ if $d = 2$ and $\deg g_i$ is odd, and $1$ otherwise. We remark that the two conditions in Theorem \ref{main1} are logically independent. For instance, taking $K = {\mathbb{Q}}$ and $d = 2$, we have that $f(z) = z^2 - \frac{k^2}{k^2 -1}$ for $k \in {\mathbb{Z}}_{\geq 2}$ not a power of two satisfies (1) but not (2), since $f^2(0) = k^2/(k^2 - 1)^2$. On the other hand, if $k \in {\mathbb{Z}}_{\geq 1}$ is odd, then $f(z) =z^2 + 2^k$ clearly fails to satisfy (1), but can be shown to satisfy (2) with $j = 0$.

Theorem \ref{main1} represents a generalization of \cite[Theorem 1.2, part (iii)]{quaddiv} in two ways. First, it holds for maps of higher degree than two, and indeed it is the first result to cover such maps. Second, it handles many values of $c$ that are not in the ring of integers of $K$; this gives for instance a partial answer to the question posed in \cite{krieger} on whether the results of \cite{quaddiv} can be extended to $z^2 + c \in {\mathbb{Q}}[z]$. The proof of Theorem \ref{main1} is made possible first by improved results on the nature of the factorization into irreducibles of iterates of $f(z)$; see Theorem \ref{evstab} and the discussion below. Part (1) of the theorem is proved via a new method that hinges on a study of the ramification degrees of extensions generated by iteration of $f$ over the local field $K_v$ given by the completion of $K$ at $v$. Part (2) of Theorem \ref{main1} is proved by a global method extending the work of the second author in \cite{quaddiv} and \cite{galmart} from certain quadratics over ${\mathbb{Q}}$ and ${\mathbb{F}}_p(t)$ to higher-degree polynomials over more general global fields. 

In the case where $K$ is a function field over ${\mathbb{F}}_q$, part (1) of Theorem \ref{main1} gives a nearly complete result. Recall that for a global field $K$ and $a \in K \setminus \{0\}$, we have the product formula
\begin{equation} \label{prodform}
\prod_{v \in M_K} |a|_v^{n_v} = 1,
\end{equation}
where $n_v $ is the degree of the local extension $[K_v : {\mathbb{Q}}_v]$ in the number field case and $[K_v : {\mathbb{F}_q}(t)_v]$ in the function field case \cite[Proposition 8.7]{Lorenzini}. Moreover, $|a|_v = 1$ for all $v \in M_K$ if and only if $a$ is a root of unity. When $K$ is a function field, we have the crucial fact that every $v \in M_K$ is non-archimedean, and the associated residue field is a finite extension of ${\mathbb{F}}_q$. Hence the residue characteristic at every place is equal to the characteristic of ${\mathbb{F}}_q$ (which is the same as the characteristic of $K$). In addition, $a \in K$ is a root of unity if and only if $a$ belongs to the algebraic closure of ${\mathbb{F}_q}$ in $K$, called the \textit{field of constants} of $K$. 
We immediately obtain:
\begin{corollary} \label{maincor1}
Let $K$ be a global function field of characteristic prime to $d$, let $f(z) = z^d + c$, and suppose that $c$ does not belong to the field of constants of $K$. 
Then $D(P_f(a_0)) = 0$ for any $a_0 \in K$.
\end{corollary}

Corollary \ref{maincor1} is a significant generalization of Theorem 1.4 of \cite{galmart}; indeed the latter essentially gives Corollary \ref{maincor1} in the special case $K = {\mathbb{F}_p}(t)$ and $f(z) = z^2 + t$, where $p$ is an odd prime. Correspondingly, in the language of \cite{galmart}, Corollary \ref{maincor1} applied to $f(z) = z^d + t$ shows that for $p \nmid d$, the hyperbolic subset  
\[\{c \in {\mathbb{C}}_p : \text{$0$ tends to an attracting cycle under iteration of $f(z) = z^d + c$}\}\]
of the $p$-adic multibrot set 
\[\{c \in {\mathbb{C}}_p : \text{$0$ has bounded orbit under iteration of $f(z)=z^d+c$}\}\]
has density zero in a natural sense. See \cite{galmart} for more details. 

We also get an interesting application of Theorem \ref{main1} in the case $K={\mathbb{Q}}(\zeta_p)$. The primes of ${\mathcal{O}}_K$ lying over the $q \in {\mathbb{Z}}$ with $q \equiv 1 \bmod{p}$ form a density one subset of the primes in ${\mathcal{O}}_K$, because these primes split, and so have norm $q$, while the norm of a prime lying over any other $q \in {\mathbb{Z}}$ is at least $q^{2}$. For a prime ${\mathfrak{q}}$ of ${\mathcal{O}}_K$, it is easy to check that ${\mathfrak{q}} \mid f^n(a_0)$ if and only if $q \mid f^n(a_0)$, where $q = {\mathfrak{q}} \cap {\mathcal{O}}_K$. 
\begin{corollary} \label{maincor2}
Let $p$ be prime and $f(z) = z^p + c$ for some $c \in {\mathbb{Z}}$ with $c \neq 0$ (if $p = 2$ we also exclude $c = -1$). Then the set of primes 
$q \equiv 1 \bmod{p}$ that belong to $P_f(a_0)$ has density zero in the set of of all primes $q \equiv 1 \bmod{p}$
\end{corollary}

Note that if $q \not\equiv 1 \bmod{p}$, then $p \nmid \#({\mathbb{Z}} / q{\mathbb{Z}})^*$, and thus $z \mapsto z^p$ is a one-to-one map on ${\mathbb{Z}} / q{\mathbb{Z}}$.  Hence $f(z) = z^p + c$ acts as a permutation on ${\mathbb{Z}} / q{\mathbb{Z}}$, and so every element of ${\mathbb{Z}} / q{\mathbb{Z}}$ is periodic under iteration of $f$.  In particular, $f^n(0) \equiv 0 \bmod{q}$ for some $n \geq 1$, and hence the density of primes in ${\mathbb{Z}}$ dividing at least one element of $O_f(0)$, which we denote $D_{\mathbb{Q}}(O_f(0))$, is at least $(p-2)/(p-1)$. This phenomenon is noted in \cite{periodsmoduloprimes} for the special case $f(z) = z^3 + 1$, where it is used  to show that $0$ may be periodic modulo a positive proportion of primes even though it is not periodic over ${\mathbb{Z}}$. Corollary \ref{maincor2} shows that in fact $D_{\mathbb{Q}}(O_f(0)) = (p-2)/(p-1)$, in particular giving $D_{\mathbb{Q}}(O_{z^3 + 1}(0)) = 1/2$. A natural extension of these considerations is to allow our initial point to be $a_0 \neq 0$. In this case Corollary \ref{maincor2} gives only $D_{\mathbb{Q}}(O_f(a_0)) \leq (p-2)/(p-1)$. In seems reasonable to expect that $D_{\mathbb{Q}}(O_f(a_0)) = 0$, but at present this appears quite difficult to prove. 

We prove Corollary \ref{maincor2} by applying condition (2) of Theorem \ref{main1}, with $j = 0$ or $j = 1$ according to whether $c$ is a $p$th power in ${\mathbb{Z}}$. See Lemma \ref{zcase}, where we verify that (2) applies in this case. In the process, we show that if $p$ is odd and $c$ is not a $p$th power in ${\mathbb{Z}}$, then $f^n(z)$ is irreducible over ${\mathbb{Q}}(\zeta_p)$ for all $n \geq 1$, and if $c$ is a $p$th power, then $f^n(z)$ has precisely $p$ irreducible factors over ${\mathbb{Q}}(\zeta_p)$  (and two irreducible factors over ${\mathbb{Q}}$) for all $n \geq 1$. This generalizes \cite[Proposition 4.5]{quaddiv}, and establishes additional cases of a conjecture of Sookdeo, namely that there are only finitely many $S$-integral points in the set $\bigcup_{n \geq 1} f^{-n}(0)$ (see \cite[Conjecture 1.2, Theorems 2.5, 2.6]{sookdeo}). 

A key ingredient in our proof of Theorem \ref{main1} is a new result giving conditions on $c$ that ensure the number of irreducible factors of $f^n(z)$ is absolutely bounded as $n$ grows. This phenomenon -- called \textit{eventual stability} -- is central to the study of arithmetic aspects of polynomial dynamics, and has attracted significant study, for instance in \cite{ostafe}, \cite{ingram}, \cite{quaddiv}, and \cite{sookdeo} (a large amount of additional work has gone into finding conditions ensuring that all iterates of $f$ are irreducible; see for example \cite{danielson}). Even over ${\mathbb{Q}}$, complicated behavior is possible; for instance, if $f(z) =  z^2-\frac{16}{9}$, then
\[f^3(z)  =  \left(z^2-2z+\frac{2}{9}\right) \left(z^2+2z+\frac{2}{9}\right) \left(z^2-\frac{22}{9}\right) \left(z^2-\frac{10}{9}\right).\]
However, for $n \geq 3$, $f^n(z)$ has precisely four irreducible factors over ${\mathbb{Q}}$ (see the remark on p. \pageref{evstabremark}).

\begin{definition} We say a polynomial $f$ is \emph{eventually stable} if there is an $N\geq 0$ and a fixed $t$ depending only on $f$ such that, for all $n>N$, $f^n$ is a product of exactly $t$ irreducible factors.
\end{definition} 

\begin{theorem} \label{evstab}
Let $d \geq 2$, let $K$ be a field of characteristic not dividing $d$, and let $f(z) = z^d + c \in K[z]$.  If there is a discrete non-archimedean absolute value on $K$ with $|c| < 1$, then $f$ is eventually stable over $K$. 
\end{theorem}

Theorem \ref{evstab} immediately yields the following corollary in the case $K = {\mathbb{Q}}$, giving another generalization of \cite[Proposition 4.5]{quaddiv} and proving the corresponding cases of Conjecture 1.2 in \cite{sookdeo}. In \cite{ingram}, Ingram proves an eventual stability-type result for polynomials over a number field, though one that is disjoint from Theorem \ref{evstab}. His methods are quite different from ours; see the discussion on p. \pageref{ingramdisc}. 

\begin{corollary} \label{evstabcor2}
Let $f(z) = z^d + c \in {\mathbb{Q}}[z]$, and suppose that $c$ is non-zero and is not the reciprocal of an integer. Then $f$ is eventually stable over ${\mathbb{Q}}$. 
\end{corollary}
The case where $c$ is the reciprocal of an integer remains open.

Theorem \ref{evstab} also allows us to obtain nearly complete information in the function field case. By a function field, we mean here something more general than a global function field: a finite extension $K$ of $F(t)$, where $F$ is any field. Function fields share the properties of global function fields mentioned above \cite[Chapter 5]{Rosen}, and we thus obtain:
\begin{corollary} \label{evstabcor}
Let $K$ be a function field of characteristic not dividing $d$, and let $f(z) = z^d + c \in K[z]$. Then $f$ is eventually stable over $K$ unless $c$ belongs to the field of constants of $K$.
\end{corollary}
When $c$ belongs to the field of constants of $K$, eventual stability need not hold. For one thing, we may have that $0$ is periodic under $f$, and hence $z \mid f^k(z)$ for some $k \geq 1$, implying that $f^{jk}(z)$ has at least $j+1$ irreducible factors. Even when $0$ is not periodic, we may not have eventual stability, particularly when the field of constants is finite. Indeed, we expect eventual stability to fail in general when $f$ is defined over a finite field, as predicted by the factorization model in \cite{settled}. An interesting example is given by $f(z) = z^2 + z + 2 \in \mathbb{F}_3[z]$, where $\mathbb{F}_3$ is the finite field with three elements. Here $0$ is not periodic under $f$, yet it can be shown that the number of distinct irreducible factors of $f^n$ is $\geq n - 1$ for all $n \geq 3$ \cite[Section 9]{ostafe}. Thus in a sense Corollary \ref{evstabcor} is best-possible in the case where $K$ is a function field over a finite field. 

We close this introduction with a sketch of the proof of Theorem \ref{main1}, which also serves as an outline for the article. We begin by showing that in both cases of Theorem \ref{main1}, there is $j \in {\mathbb{Z}}_{\geq 1}$ such that $f^j(z)=\prod_{i=1}^t g_i(z)$, with $g_i(f^n(z))$ irreducible for all $n \geq 0$. This follows from Theorems \ref{evstab} and \ref{firststab}, whose proofs are given in Section \ref{irred}. The irreducibility of the $g_i(f^n(z))$ plays a role in the results of Sections \ref{GP(f,g)} and especially those of Section \ref{max}. Now let $P_{f,g_i}(a_0)$ be the set of prime ideals ${\mathfrak{q}}$ of ${\mathcal{O}}_K$ such that ${\mathfrak{q}} \mid g_i(f^n(a_0))$ for at least one $n \geq 1$, and note that ${\mathfrak{q}}\in P_{f,g_i}(a_0)$ for some $1\leq i\leq t$ if and only if  ${\mathfrak{q}}\in P_{f}(a_0)$. In Section \ref{GP(f,g)}, we relate the density of $P_{f,g_i}(a_0)$ to Galois theory. Specifically, we recall from \cite{quaddiv} the definition of a Galois process, which furnishes an upper bound for the desired density, and in Theorems \ref{mart1} and \ref{mart2} we show that the Galois process associated to $(f, g_i)$ is an eventual martingale, and hence is a convergent stochastic process. In Section \ref{max} we use group theory and Diophantine methods (see Theorem \ref{inf many primitives}) to show that the convergence of the Galois process impiles the density of $P_{f,g_i}(a_0)$ is zero. Thus $P_{f}(a_0)$ is a finite union of zero-density sets, proving the theorem.

\section{Irreducibility Results} \label{irred}
In this section we examine irreducibility properties of polynomials of the form $g \circ f^{n}$ over a general field $K$, in the case where $f(z) = z^d + c$. 
Arithmetic properties of the \textit{translated critical orbit} $\{g(f^n(0)) : n \geq 1\}$ play a key role in this matter (see Theorem \ref{firststab}). In the event that $g \circ f^{n-1}$ is irreducible over $K$ but $g \circ f^{n}$ is not, and $K$ contains a primitive $m$th root of unity, we show that the factors of $g \circ f^{n}$ must all have a special form (Theorem \ref{splitting}). This leads to Theorem \ref{evstab}, whose proof we defer until the end of this section. We begin with a result giving arithmetic conditions on ${g(f^n(0)) : n \geq 1}$ that ensure $g \circ f^{n}$ is irreducible for all $n \geq 1$. It is a generalization of \cite[Theorem 2.2]{itconst}, and also of \cite[Proposition 1]{danielson}.
 
\begin{theorem} \label{firststab}
Let $K$ be a field of characteristic not dividing $d$, let $g, f \in K[z]$ with $f(z) = z^d + c$, and $g(z)$ monic, irreducible, and separable. Suppose that for each $n \geq 1$ the following hold:
\begin{enumerate}
	\item[(1)] $(-1)^{\epsilon}g(f^n(0))$ is not a $p$th power in $K$ for any prime $p \mid d$; and
	\item[(2)] if $4 \mid d$, then $(-1)^{\epsilon+1}4g(f^n(0))$ is not a $4$th power in $K$,
\end{enumerate}
where $\epsilon = 1$ if $n = 1$, $d$ is even, and $\deg g$ is odd, and $\epsilon = 0$ otherwise.  
Then $g \circ f^{n}$ is irreducible and separable over $K$ for all $n \geq 1$. 
\end{theorem}

\begin{remark}
In the case where $K$ is finite, one can show the theorem is if and only if. See the similar statement in 
\cite[Theorem 2.2]{itconst}. We also note that in the case where $K$ contains a primitive $d$th root of unity, the theorem holds without assuming condition (2), and one obtains a proof via taking $z = 0$ in the statement of Theorem \ref{splitting}.
\end{remark}

\begin{proof}
Let $N \geq 0$, and assume inductively that $g \circ f^{N}$ is irreducible and separable. Recall that $f^0(z) = z$, so the assumption that $g(z)$ is irreducible and separable takes care of the base case of induction.

Let $\beta$ be a root of $g \circ f^{N+1}$, and note that $\alpha:= f(\beta)$ is a root of $g \circ f^{N}$. Clearly $K(\beta)\supseteq K(\alpha)$. Now $g \circ f^{N+1}$ is irreducible if and only if 
$[K(\beta): K] = \deg(g(f^{N+1}(z)))$. However, because $g \circ f^{N}$ is irreducible, this holds if and only if 
$[K(\beta) : K(\alpha)] = d$, or in other words if $f(z) - \alpha$ is irreducible over $K(\alpha)$. Note that 
$f(z) - \alpha = z^d + c - \alpha$, and by \cite[Theorem 9.1, p. 297]{langalg} this is irreducible over $K(\alpha)$ provided that $\alpha - c$ is not a $p$th power in $K(\alpha)$ for any $p \mid d$ and, if $4 \mid d$, $-4(\alpha - c)$ is not a fourth power in $K(\alpha)$. 

We now compute:
\begin{align} 
\nonumber  N_{K(\alpha)/K}(\alpha - c) & = \prod_{g(f^{N}(\alpha')) = 0} -(c - \alpha') \\ 
\nonumber & = (-1)^{\deg g(f^N(z))} g(f^{N}(c)) \\
& = (-1)^{\deg g(f^N(z))} g(f^{N+1}(0)). \label{fin} 
\end{align} 

The first equality follows since $g \circ f^{N}$ is separable and irreducible, and thus the norm over $K$ of an element of $K(\alpha)$ is the product of its Galois conjugates, and every root of $g \circ f^{N}$ is a Galois conjugate of $\alpha$. The second equality follows because $g \circ f^{N}$ is monic.

If $d$ is odd, then $-1$ is a $p$th power in $K$ for each prime $p \mid d$, and hence condition (1) with $\epsilon = 0$ implies the expression in \eqref{fin} is not a $p$th power in $K$. The multiplicativity of the norm map then gives that $\alpha - c$ is not a $p$th power in $K(\alpha)$, showing that $g \circ f^{N+1}$ is irreducible over $K$. If $d$ is even and $\deg (g \circ f^{N})$ is even, then similarly to the case where $d$ is odd, conditions (1) and (2) with $\epsilon = 0$ give the irreducibility of $g \circ f^{N+1}$ over $K$. When $d$ is even, $\deg(g \circ f^{N})$ can only be odd when $\deg g$ is odd and $N = 0$, in which case we require $\epsilon = 1$ in conditions (1) and (2) to ensure the irreducibility of $g \circ f^{N+1}$. 

Finally, each root of $g \circ f^{N+1}$ is a root of $z^d + c - \alpha$ for some root $\alpha$ of $g \circ f^{N}$. Because $K$ has characteristic not dividing $d$, it follows that $z^d + c - \alpha$ is separable provided that $c - \alpha \neq 0$. But the latter is impossible since otherwise \eqref{fin} gives $(-1)^{\epsilon} g(f^{N+1}(0)) = 0$, contrary to the hypothesis of the lemma. But $g \circ f^{N}$ is separable by inductive hypotheses, and this shows $g \circ f^{N+1}$ is separable.
\end{proof}

We now embark on a sequence of results that leads to the proof of Theorem \ref{evstab}.

\begin{lemma}\label{Capelli Part} Let $d \geq 2$, let $L$ be a field of characteristic not dividing $d$, and let $\zeta_d \in L$ be a primitive $d$th root of unity. Suppose that $a \in L$, $a \neq 0$, and let $E$ be the splitting field of $z^d -a$ over $L$. Then every orbit of the action of ${\rm Gal\,}(E/L)$ on the roots of $z^d - a$ has the form
\begin{equation} \label{orbit}
	\{\zeta_d^{rm} \beta : r = 1, \ldots, d/m\}
\end{equation}
for some $m \mid d$, where $\beta$ may be taken to be any element of the orbit.
\end{lemma}

\begin{proof}
Note that ${\rm char}(L) \nmid d$ and $a \neq 0$ ensure that $z^d - a$ is separable over $L$, and hence if $\beta_0$ is any root of $z^d - a$, then the full set of roots is $\{\beta_0, \zeta_d\beta_0, \ldots, \zeta_d^{d-1}\beta_0\}$. Consider an orbit $O$ of the action of ${\rm Gal\,}(E/L)$, and choose some $\beta \in O$. Let $k$ be the least positive integer such that $\beta^k\in L$. Certainly, $k\leq d$.  We claim that $k$ is a divisor of $d$.  Let $m\in\mathbb{Z}$ be such that $0\leq d-mk<k$.  Then, as $\beta^d=\beta^{mk}\beta^{d-mk}$, it must be that $d-mk=0$, because $\beta^{d-mk}\in L$, but $d-mk<k$. 

Let $m=\frac{d}{k}$, and put
\[s(z) = \prod_{r = 1}^{d/m} \left(z - (\zeta_d^m)^r\beta \right) = z^k - \beta^k \in L[z].\]
If $s(z)$ has a non-trivial factor $t(z)$ over $L[z]$, then $(\zeta_d^m)^u \beta^v = t(0) \in L$ for some integer $u$ and
some $0 < v < k$.  Hence $\beta^v \in L$, contradicting the minimality of $k$. Thus $s(z)$ is irreducible over $L$, proving that $O$ has the form \eqref{orbit}. 
\end{proof}

\begin{theorem}\label{splitting}
Let $d \geq 2$, let $L$ be a field of characteristic not dividing $d$, and let $\zeta_d \in L$ be a primitive $d$th root of unity.  Let $f(z)=z^d+c \in L[z]$ and let $g(z) \in L[z]$ be monic and separable. Take $f^0(z) = z$, and suppose that $g \circ f^{n-1}$ is irreducible over $L$ for some $n \geq 1$. If $g \circ f^{n}$ has a non-trivial factorization over $L$, then we have
\begin{equation} \label{fact}
	g(f^n(z)) = (-1)^{\epsilon} \prod_{k=1}^{m} h(\zeta_d^kz),
\end{equation}
where $h(z) \in L[z]$ is irreducible, $m \mid d$, $m \geq 2$, $\epsilon = 1$ if $\deg(g \circ f^{n-1})$ is odd and $m$ is even, and $\epsilon = 0$ otherwise. 
\end{theorem}

\begin{proof}
Let $E$ be the splitting field of $g \circ f^{n}$ over $L$.  Let $G = {\rm Gal\,}(E/L)$, and let $\alpha_1, \ldots, \alpha_j$ be the roots of $g \circ f^{n-1}$ in $E$.  Consider the $G$-orbit $O(\beta)$ of a root $\beta$ of $g \circ f^{n}$. Without loss of generality, say $f(\beta) = \alpha_1$, and so $\beta$ is a root of $f(z) - \alpha_1 = z^d - (\alpha_1 - c)$.  Now $G$ has the subgroup $S := {\rm Gal\,}(E/L(\alpha_1))$.  Because $g \circ f^{n-1}$ is irreducible, the action of $G$ on the $\alpha_i$ is transitive, and hence we may choose $\sigma_1, \ldots, \sigma_j \in G$ such that $\sigma_i(\alpha_1) = \alpha_i$, or in other words $\sigma_i(\beta)$ is a root of $f(z) - \alpha_i$ for $i = 1, \ldots, j$. 

As $G_n$ is the disjoint union of the cosets $\sigma_iS$, so $O(\beta)$ is the disjoint union of the sets $\{\sigma_i s(\beta) : s \in S\}$.  Lemma \ref{Capelli Part} then gives
\[\{s(\beta) : s \in S\} = \{\zeta_d^{rm} \beta : r = 1, \ldots, d/m\}\]
for some divisor $d$ of $m$, and thus
\[\{\sigma_is(\beta) : s \in S\} = \{\zeta_d^{rm} \sigma_i(\beta) : r = 1, \ldots, d/m\},\]
We now put
\[h(z) := \prod_{i = 1}^j \prod_{r = 1}^{d/m} \left(z - \zeta_d^{rm} \sigma_i(\beta) \right),\]
which is an irreducible element of $L[z]$ since its roots consist of a full $G$-orbit.  Moreover, because $g \circ f^{n-1}$ is irreducible, every root of $g \circ f^{n}$ may be written
\begin{equation} \label{roots}
	\zeta_d^{rm - k} \sigma_i(\beta),
\end{equation}
for some $k$ with $1 \leq k \leq d$ and some $i$ with $1 \leq i \leq j$.  Note that 
\[h(\zeta_d^kz) = \prod_{i = 1}^j \prod_{r = 1}^{d/m} \left(\zeta_d^kz - \zeta_d^{rm} \sigma_i(\beta) \right) = \left(\zeta_d^{dj/m} \right)^k \prod_{i = 1}^j \prod_{r = 1}^{d/m} \left(z - \zeta_d^{rm - k} \sigma_i(\beta) \right).\]
Taking the product over $k = 1, \ldots, m$ and using \eqref{roots} gives
\[\prod_{k=1}^{m} h(\zeta_d^kz) = \left( \prod_{k = 1}^m \left(\zeta_d^{dj/m} \right)^k \right) g(f^n(z)) = \zeta_d^{dj(m-1)/2} g(f^n(z)),\]
where the first equality follows since $g \circ f^{n}$ is monic. Note that $j = \deg(g \circ f^{n-1})$. 
\end{proof}

\begin{definition}\label{RDS} 
Let $A=\{a_i\}_{i \geq 1}$ be a sequence in a field $K$. We say $A$ is a \emph{rigid divisibility sequence over $K$} if for each non-archimedean absolute value $|\cdot|$ on K, the following hold:
\begin{itemize}
	\item[(1)] If $|a_n| < 1$, then $|a_n| = |a_{kn}|$ for any $k \geq 1$.
	\item[(2)] If $|a_n| < 1$ and $|a_j| < 1$, then $|a_{\gcd(n, j)}| < 1$.
\end{itemize}
\end{definition}

Recall that when $A$ is a sequence of rational integers, it is a \textit{divisibility sequence} when $a_n \mid a_m$ whenever $n \mid m$; this condition is ensured by (and strictly weaker than) condition (1) in Definition \ref{RDS}. Additionally, it is a \textit{strong divisibility sequence} when $\gcd (a_n, a_m) = a_{\gcd(m,n)}$, which is ensured by condition (2). Hence every rigid divisibility sequence is also a strong divisibility sequence, though the converse is false. A consequence of Definition \ref{RDS} is that if $|a_n| < 1$ and $|a_k| < 1$ for some non-archimedean absolute value, then $|a_n| = |a_k|$. Rigid divisibility sequences arise naturally from iteration of certain polynomials, and they have proved useful in analyzing arithmetic phenomena such as primitive divisors \cite{doerksen, krieger, rice}.  In Lemma \ref{rds1} we generalize  \cite[Lemma 4]{doerksen}, \cite[Lemma 5.3]{quaddiv}, and \cite[Lemma 2.3]{krieger}, where consideration is restricted to $c$ belonging to ${\mathbb{Z}}$ or ${\mathbb{Q}}$.  

Recall that for a non-archimedean absolute value $|\cdot|$ on $K$, the set $\{x \in K : |x| \leq 1\}$ is a ring, and $\{x \in K : |x| < 1\}$ is its unique maximal ideal. The associated quotient field is called the residue field.  

\begin{lemma} \label{rds1} 
Let $K$ be a field and $f(z)=z^d+c\in K[z]$ for some $d\geq 2$.   Then $\{f^n(0)\}_{n \geq 1}$ is a rigid divisibility sequence over $K$.
\end{lemma}

\begin{remark}
One can further generalize Lemma \ref{rds1} to the case where $f(z)$ has no linear term, but at the price of excluding certain absolute values of $K$ from the Definition \ref{RDS}. For instance, consider $f(z) = z^3 + (1/27)z^2 + 3$, and note that $\{f^n(0)\}_{n \geq 1}$ is a rigid divisibility sequence for all non-archimedean absolute values on ${\mathbb{Q}}$ except the $3$-adic absolute value.  The interested reader should also consult \cite[Proposition 3.5]{rice}, which gives a slightly stronger conclusion than that of Lemma \ref{rds1}. One can further generalize such a result to sequences of the form $f^n(\gamma) - \gamma$, where $\gamma$ is a critical point of $f(z)$.
\end{remark}

\begin{proof}  Let $|\cdot|$ be a non-archimedean absolute value on $K$. We begin with the observation that either $|f^n(0)| > 1$ for all $n \geq 1$ or $|f^n(0)| \leq 1$ for all $n \geq 1$. Indeed, assume that $|f^n(0)| > 1$ for some $n \geq 1$ and without loss let $n$ be minimal with this property. Write $c = f^n(0) - (f^{n-1}(0))^d$, taking $f^0(x) = x$ in the case $n = 1$. Then $|c| = |f^n(0)|$ by the ultrametric property. Therefore $|f(0)| = |c| > 1$, whence $n = 1$. We now have $|f^i(0)| = |f^{i-1}(0)|^m > |f^{i-1}(0)|$ for each $i \geq 2$, proving that $|f^n(0)| > 1$ for all $n \geq 1$. 

To prove property (1), suppose that $|f^n(0)| < 1$. By the previous paragraph, this gives $1 \geq |f(0)| = |c|$. We induct on $k$, noting first that if $k=1$, then trivially $|f^{kn}(0)| = |f^n(0)|$. Suppose that $|f^{(k-1)n}(0)| = |f^n(0)| < 1$. Write $f^n(z)=f^n(0)+\sum_{i=1}^{d^{n-1}} c_iz^{di}$, and note that the $c_i$ are elements of ${\mathbb{Z}}[c]$ and thus $|c_i| \leq 1$ because $|c| \leq 1$.  Observe that $f^{kn}(0)=f^n(f^{(k-1)n}(0))$ implies
\[|f^{kn}(0)| = \left| f^n(0)+\sum_{i=1}^{d^{n-1}} c_i\left(f^{(k-1)n}(0)\right)^{di}\right| = |f^n(0)|,\]
where the last equality follows from the fact that $|f^{(k-1)n}(0)|^{di} < |f^n(0)|$ for $i \geq 1$.  

To prove property (2), assume $|f^m(0)| < 1$ for some $m \geq 1$, and let $m$ be the minimal positive integer with this property. By the argument at the beginning of the proof of the lemma, $|f^n(0)| \leq 1$ for all $i$. Therefore the sequence $\overline{0}, \overline{f(0)}, \overline{f^2(0)}, \ldots$ in the residue field of $K$ is a cycle containing precisely $m$ distinct elements. Hence for any $j=\ell m+r$, $\overline{f^k(0)}$ will be zero if and only if $r=0$. So if both $|f^j(0)| < 1$ and $|f^n(0)| < 1$, then $m\mid\gcd(j,n)$, yielding property (2).
\end{proof}

\begin{proof}[{Proof of Theorem \ref{evstab}}]
We begin by choosing an extension of $|\cdot|$ to $K(\zeta_d)$; any such extension will still be discrete, since $K(\zeta_d)/K$ is a finite extension. We now replace $K$ by $K(\zeta_d)$, noting that Lemma \ref{rds1} shows that $\{f^n(0)\}_{n \geq 1}$ is still a rigid divisibility sequence over this larger field. As in the proof of Lemma \ref{rds1}, the assumption that $|c| < 1$ gives that all roots of iterates of $f$ have absolute value at most 1, and hence the same holds for the coefficients of any divisor of an iterate of $f$. 

Let $\{g_1,g_2,\dots\}$ be a (possibly finite) sequence of irreducible polynomials in $K[z]$ and $\{n_1, n_2, \ldots \}$ a sequence of positive integers with the following properties: $g_1$ properly divides $f^{n_1}$ while $f^{n_1-1}$ is irreducible, and for $i \geq 2$, $g_i$ properly divides $g_{i-1}\circ f^{n_i}$ while $g_{i-1}\circ f^{n_i-1}$ is irreducible.  To prove the theorem, we show that any such sequence must be finite.

By Theorem \ref{splitting}, we have $f^{n_1}(0) = \pm g_1(0)^d$ for some $d > 1$, and because $|f^{n_1}(0)| = |c|$ by Lemma \ref{rds1}, we have $|g_1(0)| = |c|^{1/d}$, showing that $1 > |g_1(0)| > |c|$. Assume that, for some $i \geq 2$, $g_i$ is defined and
\begin{equation} \label{chain}
1 > |g_{i-1}(0)| > \cdots > |g_1(0)| > |c|.
\end{equation}
By Theorem \ref{splitting}, $g_{i-1}(f^{n_i}(0)) = \pm g_i(0)^d$ for some $d > 1$, so that $1 > |g_i(0)| > |g_{i-1}(f^{n_i}(0))|.$  Now the coefficients of $g_i(x)$ are integral, while by Lemma \ref{rds1} and the inductive hypothesis we have $|f^{n_i}(0)| = |c| < |g_{i-1}(0)| < 1$. Therefore the sum $g_{i-1}(f^{n_i}(z))|_{z = 0}$ is dominated by the term $g_{i-1}(0)$, so $|g_{i-1}(f^{n_i}(0))| = |g_{i-1}(0)|$.  

We have thus shown that every element of the sequence $\{g_1, g_2, \ldots\}$ fits into a chain of the form \eqref{chain}. Because $|\cdot|$ is discrete, any such chain must have finite length, proving the theorem. 
\end{proof}

\begin{remark} \label{evstabremark}
Theorem \ref{evstab} in fact gives a quantitative result. Let $v$ be the normalized valuation associated to $|\cdot|$, so that $v(K^*) = {\mathbb{Z}}$ (see e.g. \cite[II.3]{neukirch}), and let $i_f$ be the limit as $n$ grows of the number of irreducible factors of $f^n(x)$.  If $v(c) = e$ (which by assumption is positive), then every sequence $\{g_1,g_2,\dots\}$ as in the proof of Theorem \ref{evstab} has length at most $\log_2 e$ (note much better bounds are possible for specific $d$). By Theorem \ref{splitting} we then have $i_f \leq d^{\log_2 e}$. It would be very interesting to have a uniform bound for $i_f$ for some given family $z^d + c, c \in K$. Some work has been done in this direction when $K = {\mathbb{Q}}$ and $f(z) = z^2 + c$; in \cite{faber-hutz} it is shown that no iterate of $f(z)$ has more than 6 linear factors over ${\mathbb{Q}}$, assuming certain standard conjectures on $L$-series. However, as noted in the discussion after Corollary \ref{evstabcor2}, eventual stability has not even been fully established for this family.  
\end{remark}

\section{The Galois Process and Related Results}\label{GP(f,g)}

In this section we connect the problem of determining the densities of sets of primes dividing orbits of $z^d+c$ in our main results to the Galois theory of iterates of $f$.  We
recall from \cite{quaddiv} the definition of the Galois process attached to a pair $(f, g)$ of polynomials. 

Let $K$ be a field, and let $f(z), g(z) \in K[z]$.  We fix an algebraic closure $\overline{K}$ of $K$ and let $T_n$ denote the set of roots of $g \circ f^{n}$ in $\overline{K}$, $K_n = K(T_n)$ be the splitting field of $g \circ f^{n}$, and $G_n = {\rm Gal\,}(K_n/K)$.  (We will use this notation for the remainder of the paper.)
Let $G_{\infty} =\varprojlim G_n$, and take $\mu$ to be a Haar measure on $G_{\infty}$ with $\mu(G_{\infty})=1.$  For $\sigma\in G_{\infty}$, let $\pi_n(\sigma)$ be the restriction of $\sigma$ to $G_n$. (For a more detailed exposition, see the remark in Section \ref{tree defs}.)
We are interested in how the proportion of elements of $G_n$ fixing at least one $\beta \in T_n$ varies with $n$. We define functions $Y_n : G_\infty \to {\mathbb{Z}}$ by 
\begin{equation} \label{gpdef}
Y_n(\sigma) = \#\{\text{fixed points of $\pi_n(\sigma)$ acting on $T_n$}\}.
\end{equation}

Because $\mu$ is a probability measure on $G_\infty$, the $Y_n$ are in fact random variables, and hence the sequence $Y_1, Y_2, \ldots$ is a stochastic process, which we refer to as the {\em Galois process} of $(f,g)$. We denote by $E(Y)$ the expected value of the random variable $Y$. Note that because $\mu(\pi_i^{-1}(S)) = \#S/\#G_i$ for any $S \subseteq G_i$, we have that $\mu(Y_1 = t_1, \ldots, Y_n = t_n)$ is given by 
\begin{equation} \label{fpchar}
	\frac{1}{\# G_n}\# \left\{\sigma \in G_n : \mbox{$\sigma$ fixes $t_i$ elements of $T_i$ for $i = 1,2, \ldots, n$} \right\}.
\end{equation}

The connection between the Galois process and our main results is given by \cite[Theorem 2.1]{quaddiv} and the remarks following.  We state here a version applicable to our present considerations:
\begin{theorem}{\cite[Theorem 2.1]{quaddiv}} \label{densecon}
Let $f,g\in  K[z]$ be polynomials with $g \circ f^{n}$ separable for all $n$. Let $a_n = g(f^n(a_0))$ with $a_0\in K$.  Then the density of primes dividing at least one $a_n$ is bounded above by 
\[\lim_{n \to \infty} \mu(Y_n > 0),\]
where $Y_n$ is the $n$th random variable in the Galois process of $f,g$. 
\end{theorem}

While \cite[Theorem 2.1]{quaddiv} is stated for $f, g \in {\mathbb{Z}}[z]$, it trivially extends to $f, g \in {\mathcal{O}}_K[z]$, and may be extended to $f, g \in K[z]$ by excluding the finitely many primes of ${\mathcal{O}}_K$ at which at least one coefficient of $f$ or $g$ has negative valuation. 

\begin{definition}
A stochastic process with probability measure $\mu$ and random variables $Y_1, Y_2, \ldots$ taking values in $\mathbb{R}$ is a {\em martingale} if for all $n \geq 2$ and any $t_i $, 
\[E(Y_n \mid Y_{1} = t_{1}, Y_2 = t_2, \ldots, Y_{n-1} = t_{n-1}) = t_{n-1},\]
provided $\mu(Y_{1} = t_{1}, Y_2 = t_2, \ldots, Y_{n-1} = t_{n-1}) > 0$.  We call $Y_1, Y_2, \ldots$ an {\em eventual martingale} if for some $N \geq 1$ the process $Y_N, Y_{N+1}, Y_{N+2}, \ldots$ is a martingale.
\end{definition}

We prove two main results in this section, namely:

\begin{theorem} \label{mart1}
Suppose that $d \geq 2$ and $K$ is a global field of characteristic not dividing $d$ and containing a primitive $d$th root of unity. Let $f(z) = z^d + c \in K[z]$, $g(z) \in K[z]$ divide an iterate of $f$, and suppose that there is a place ${\mathfrak{p}}$ of $K$ whose residue characteristic is prime to $d$ and such that $v_{\mathfrak{p}}(c) > 0$. Then the Galois process associated to $(f, g)$ is an eventual martingale. 
\end{theorem}

\begin{theorem} \label{mart2}
Suppose that $d$ is prime and $K$ is a global field of characteristic not dividing $d$ and containing a primitive $d$th root of unity. Let $f(z) = z^d + c \in K[z]$, and let $g(z) \in K[z]$ divide an iterate of $f$. Assume that for $n \geq 1$, $(-1)^{\epsilon} g(f^n(0))$ is not a $d$th power in $K$, where $\epsilon = 1$ if $n = 1$, $d =2$, and $\deg g$ is odd, and $\epsilon = 0$ otherwise. Then the Galois process associated to $(f, g)$ is a martingale. 
\end{theorem}

These two theorems correspond to cases (1) and (2) of Theorem \ref{main1}. While there are many cases covered by both Theorems \ref{mart1} and \ref{mart2}, greater generality can be achieved by using both. For example,  the case where $g(z) = f^0(z) = z$, $f(z) = z^6 + 5$, and $K = {\mathbb{Q}}(\zeta_6)$ is covered by Theorem \ref{mart1}, and $g(z) = z$, $f(z) = z^3 + 3$, and $K = {\mathbb{Q}}(\zeta_3)$ by Theorem \ref{mart2}, but neither theorem covers both. The proof of  Theorem \ref{mart2} is substantially more involved than that of Theorem \ref{mart1}.

\subsection{Local theory and proof of Theorem \ref{mart1}}
To prove Theorem \ref{mart1}, it is enough by \cite[Theorem 2.5]{quaddiv} to show that for sufficiently large $n$ and any root $\alpha$ of $g \circ f^{n-1}$, the polynomial $f(z) - \alpha$ is irreducible over the splitting field $K_{n-1}$ of $g \circ f^{n-1}$. This is equivalent to 
\begin{equation} \label{degree}
	[K_{n-1}(\beta) : K_{n-1}] = d, 
\end{equation} 
for any root $\beta$ of $g \circ f^{n}$. 

Denote by $K_{\mathfrak{p}}$ the completion of $K$ at the prime ${\mathfrak{p}}$.  Fix an embedding $\omega$ of $\overline{K}$ into $\overline{K_{\mathfrak{p}}}$, and by abuse of notation we denote by $L_{\mathfrak{p}}$ the completion of $\omega(L)$, for any extension $L$ of $K$. Our strategy for showing \eqref{degree} is to prove the stronger statement $[(K_{n-1}(\beta))_{\mathfrak{p}} : (K_{n-1})_{\mathfrak{p}}] = d$, which we accomplish by showing that the ramification degree of $(K_{n-1}(\beta))_{\mathfrak{p}}$ over  $(K_{n-1})_{\mathfrak{p}}$ is $d$. The extensions involved are compositions of certain Kummer extensions, whose ramification degrees are described in the following lemma. 

\begin{lemma} \label{Kummer ramification}
Let $L$ be a field that is complete with respect to a discrete valuation $v$. Suppose that $d \geq 2$, the residue characteristic of $L$ is prime to $d$, and $L$ contains a primitive $d$th root of unity. If $a \in L$ with $v(a) = r \geq 0$, then for any root $\rho$ of $z^d - a$, the ramification degree of $L(\rho)$ over $L$ is $d/\gcd(d,r)$. 
\end{lemma}

\begin{remark}
Lemma \ref{Kummer ramification} holds regardless of whether $z^d - a$ is irreducible over $L$. This plays a key role in the proof of Theorem \ref{mart1}.
\end{remark}

\begin{proof}
Let $m = \gcd(d,r)$ and $\pi$ be a uniformizer for $L$, so that $a = u\pi^r$ for some $u$ with $v(u) = 0$.  First assume that $z^d - a$ is irreducible over $L$. 
The Newton polygon of $z^d - a$ consists of a segment of slope $r/d$, and hence $v(\rho) = r/d = (r/m)/(d/m)$, where the latter fraction is in lowest terms. It follows that $L(\rho)$ has ramification degree at least $d/m$. On the other hand, 
\[\left( \frac{\rho^{d/m}}{\pi^{r/m}} \right)^m = u,\]
and $z^m - u$ must be irreducible over $L$, for otherwise $\pi^r[(z^{(d/m)}/\pi^{(r/m)})^m - u] =z^d - a$ has a non-trivial factorization, contradicting our assumption. By Hensel's Lemma and the fact that $d$ (and hence $m$) is prime to the residue characteristic of $L$, we have that $z^m - \overline{u}$ is irreducible over the residue field of $L$. Thus $(\rho^{d/m})/(\pi^{r/m})$ generates an unramified sub-extension of $L(\rho)$ of degree $m$, proving that the ramification degree of $L(\rho)$ is exactly $d/m$. 

Suppose now that $z^d - a$ is not necessarily irreducible, and let $\rho$ be a root. By the proof of Lemma \ref{Capelli Part}, $\rho$ is a root of an irreducible polynomial of the form $z^k - \rho^k \in L[z]$ for some $k \mid d$. By the previous paragraph, $L(\rho)$ has ramification degree $k/\gcd(k,v(\rho^k))$ over $L$. However, $k = (k/d) \cdot d$ and $v(\rho^k) = (k/d) \cdot r$, whence $\gcd(k, v(\rho^k)) = (k/d) \cdot m$. Therefore $L(\rho)$ has ramification degree $[(k/d) \cdot d]/[(k/d) \cdot m] = d/m$, as desired. 
\end{proof}

\begin{proof}[Proof of Theorem \ref{mart1}]
Let $r = v_{\mathfrak{p}}(g(0))$ and $d_n = \deg(g \circ f^{n})$. We first use our assumption that $g$ divides an iterate of $f$ to show that $v_{\mathfrak{p}}(\beta) = r/d_n$ for any root $\beta$ of $g \circ f^{n}$. It is straightforward to show that, for any $k\geq 0$, the non-leading coefficients of $f^k(z)$ are polynomials in $c$ without constant coefficients, and moreover by Lemma \ref{rds1} we have that $v_{\mathfrak{p}}(f^k(0)) = v_{\mathfrak{p}}(c)$. The Newton polygon of $f^k(z)$ is thus a single line segment of slope $-v_{\mathfrak{p}}(c)/d_n$. In the case where $g$ is an iterate of $f$, we have $v_{\mathfrak{p}}(g(0)) = v_{\mathfrak{p}}(c)$, and hence all roots of $g \circ f^{n}$ have ${\mathfrak{p}}$-adic valuation $r/d_n$.  If $g$ is not an iterate of $f$, we may apply Theorem \ref{splitting}; taking $z = 0$ there implies that if $h(z)$ is any divisor of an iterate of $f$, then $0 < v_{\mathfrak{p}}(h(0)) < v_{\mathfrak{p}}(c)$. Apply this to the present $g$ (which we remark plays the role of $h$ in Theorem \ref{splitting}) to get $0 < v_{\mathfrak{p}}(g(0)) < v_{\mathfrak{p}}(c)$. The ultra-metric inequality then gives that $v_{\mathfrak{p}}(g(f^n(0))) = v_{\mathfrak{p}}(g(0))$ for all $n \geq 1$. Hence the Newton polygon of $g \circ f^{n}$ is a single segment of slope $-r/d_n$, as desired. 

Denote by $e(L_{\mathfrak{p}})$ the ramification degree of an extension $L_{\mathfrak{p}}$ of $K_{\mathfrak{p}}$. Let $\beta_{n-1}$ be a root of $g \circ f^{n-1}$ and fix another root $\beta_{n-1}'$. We have $v_{\mathfrak{p}}(\beta_{n-1}) = v_{\mathfrak{p}}(\beta_{n-1}')$ by the previous paragraph, and hence $e(K(\beta_{n-1})_{\mathfrak{p}}) = e(K(\beta_{n-1}')_{\mathfrak{p}})$ by Lemma \ref{Kummer ramification}. It follows from Theorem \ref{splitting} that any irreducible factor of $g \circ f^{n-1}$ over $K_{\mathfrak{p}}$ has degree dividing $\deg (g \circ f^{n-1})$, which in turn divides a power of $d$. Applying this to the minimal polynomials over $K_{\mathfrak{p}}$ of $\beta_{n-1}$ and $\beta_{n-1}'$, we see that $e(K(\beta_{n-1})_{\mathfrak{p}})$ and $e(K(\beta_{n-1}')_{\mathfrak{p}})$ are prime to the residue characteristic of $K_{\mathfrak{p}}$, so that both extensions are tamely ramified. By Abhyankar's lemma \cite[Theorem 3]{Cornell}, we have 
\[e(K(\beta_{n-1}, \beta_{n-1}')_{\mathfrak{p}}) = \gcd\left(e(K(\beta_{n-1})_{\mathfrak{p}}), e(K(\beta_{n-1}')_{\mathfrak{p}}) \right) = e(K(\beta_{n-1})_{\mathfrak{p}}).\]
Applying this argument repeatedly, we have 
\begin{equation} \label{noram}
e((K_{n-1})_{\mathfrak{p}}) = e(K(\beta_{n-1})_{\mathfrak{p}}).
\end{equation} 

Let $\beta_1, \beta_2, \ldots$ be such that $\beta_n$ is a root of $g \circ f^{n}$ and $f(\beta_{n}) = \beta_{n-1}$ for all $n \geq 2$. Put $e_n = e(K(\beta_n)_{\mathfrak{p}})$. Because $v_{\mathfrak{p}}(\beta_n) = r/d_n$ and the value group of $K(\beta_n)_{\mathfrak{p}}$ is $(1/e_n){\mathbb{Z}}$, we have that $r/d_n$ is a multiple of $1/e_n$. Consider the sequence of positive integers $\{k_n\}_{n \geq 1}$ such that 
\[\frac{r}{d_n}=\frac{k_n}{e_n}.\]
Then we have
\[1=\frac{k_{n}d_{n}e_{n-1}}{k_{n-1}d_{n-1}e_{n}},\]
and therefore
\[d=\frac{d_{n}}{d_{n-1}} = \left(\frac{e_{n}}{e_{n-1}}\right)\left(\frac{k_{n-1}}{k_{n}}\right).\]
As $(e_{n}/e_{n-1})$ divides $[K(\beta_{n})_{\mathfrak{p}} : K(\beta_{n-1})_{\mathfrak{p}}]$, which in turn divides $d$,  we must have $k_{n}\mid k_{n-1}$, with moreover $k_{n}=k_{n-1}$ if and only if $(e_{n}/e_{n-1})=d$. Because $k_1$ is fixed, there is some $n_0$ such that $n>n_0$ implies $e_{n}/e_{n-1}=d$. Thus we have 
\begin{equation} \label{altevstab}
	\text{$e(K(\beta_{n})_{\mathfrak{p}}) = d \cdot e(K(\beta_{n-1})_{\mathfrak{p}})$ \qquad for $n > n_0$,}
\end{equation} 
and because $e(K(\beta_{n})_{\mathfrak{p}})$ is identical for all roots $\beta_n$ of $g \circ f^{n}$, $n_0$ does not depend on the choice of $\beta_n$. 

From \eqref{noram}, we now obtain $e(K(\beta_{n})_{\mathfrak{p}}) = d \cdot e((K_{n-1})_{\mathfrak{p}})$ for $n > n_0$. Because  
\[e(K(\beta_{n})_{\mathfrak{p}}) \leq e((K_{n-1}(\beta_{n}))_{\mathfrak{p}}) \leq d \cdot e((K_{n-1})_{\mathfrak{p}}),\]
where the last inequality follows since $[(K_{n-1}(\beta_{n}))_{\mathfrak{p}} : (K_{n-1})_{\mathfrak{p}}] \leq d$,
we have shown $e((K_{n-1}(\beta_{n}))_{\mathfrak{p}}) = d \cdot e((K_{n-1})_{\mathfrak{p}})$. This proves $[(K_{n-1}(\beta_n))_{\mathfrak{p}} : (K_{n-1})_{\mathfrak{p}}] = d$ for $n > n_0$. The argument applies to any root $\beta_n$ of $g \circ f^{n}$, thus establishing \eqref{degree} for $n > n_0$.
\end{proof}

\begin{remark}
From \eqref{altevstab} it follows that $[K(\beta_{n})_{\mathfrak{p}} : K(\beta_{n-1})_{\mathfrak{p}}] = d$ for $n > n_0$ and all roots $\beta_n$ of $g \circ f^{n}$. This gives an alternate proof of Theorem \ref{evstab}.  
\end{remark}

\subsection{Proof of Theorem \ref{mart2}: background and definitions}\label{tree defs}
Recall that $G_n$ is the Galois group of $K_n = K(T_n)$ over $K$, where $T_n$ is the set of roots of $g \circ f^{n}$.  A key property of the action of $G_n$ on $T_n$ is that it must commute with the natural map $f : T_n \to T_{n-1}$. We thus introduce some terminology relevant to such group actions.

If $G$ is a group, recall that a $G$-set is any set $S$ on which $G$ acts, and a map $\phi : S \to S'$ is a morphism of $G$-sets if $\phi(\sigma(s)) = \sigma(\phi(s))$ for all $\sigma \in G$ and $s \in S$.  A {\em fiber system} on a $G$-set $S$ is the set of fibers of any morphism $\phi : S \to S'$ of $G$-sets.  It is easy to check that a partition ${\mathcal{S}}$ of $S$ is a fiber system if and only if $\sigma(T) \in {\mathcal{S}}$ for each $T \in {\mathcal{S}}$, or in other words the constituent sets of ${\mathcal{S}}$ are permuted by the action of $G$.  For a set $S$ and a partition ${\mathcal{S}}$ of $S$, denote by ${\rm Perm}(S, {\mathcal{S}})$ the set of all permutations of $S$ that act as permutations on ${\mathcal{S}}$. Note that if $G$ acts on $S$ and ${\mathcal{S}}$ is a fiber system for the $G$-set $S$, then $G \leq {\rm Perm}(S, {\mathcal{S}})$.
Suppose that ${\mathcal{S}} = \{S_1, \ldots, S_k\}$ and each $S_i$ has $d$ elements. Fix a permutation $\sigma_{\mathcal{S}} \in {\rm Sym}(S)$ whose orbits are precisely the sets $S_i$, and fix a distinguished element $s_i$ in each $S_i$; this is equivalent to fixing an ordering of the elements of each $S_i$. Now each $\tau \in {\rm Perm}(S, {\mathcal{S}})$ induces a permutation $\tau'$ on ${\mathcal{S}}$. Moreover, if $\tau(S_i) = S_j$, then an element $\delta_i \in {\rm Sym}(d)$, the symmetric group on $d$ letters, is determined as follows: put $\delta_i(\ell_1) = \ell_2$ if 
\begin{equation} \label{restrictiondef}
	\tau(\sigma_{\mathcal{S}}^{\ell_1}(s_i)) = \sigma_{\mathcal{S}}^{\ell_2}(s_j).
\end{equation}
We thus obtain a map 
\begin{align} \label{wreathisom}
	\Phi :  {\rm Perm}(S, {\mathcal{S}})  & \to {\rm Sym}(d) \wr {\rm Sym}({\mathcal{S}}) \\
	 \tau \qquad  & \mapsto ((\delta_1, \ldots, \delta_k), \tau') \nonumber
\end{align}
that is readily seen to be an isomorphism. Recall that the wreath product ${\rm Sym}(d)~\wr~{\rm Sym}({\mathcal{S}})$ is the semi-direct product ${\rm Sym}(d)^{|{\mathcal{S}}|} \rtimes {\rm Sym}({\mathcal{S}})$ with the natural action of ${\rm Sym}({\mathcal{S}})$ on indices, i.e. 
\[((\delta_1, \ldots, \delta_k), \tau') \cdot ((\epsilon_1, \ldots, \epsilon_k), \omega') = ((\delta_1\epsilon_{\tau'(1)}, \ldots,  \delta_k\epsilon_{\tau'(k)}), \tau' \omega'),\]
where we say $\tau'(1) = j$, when $\tau'(S_1) = S_j$.  We refer to the permutation $\delta_i$ as the \textit{restriction} of $\tau$ to the index $i$, and often write it $\tau|_i$. Note that it depends not only on $\tau$ and $i$, but also on our choices of $\sigma_{\mathcal{S}}$ and the $s_i$. A useful map is given by taking the product of the restrictions: 
\begin{align}  \label{psimap}
	\psi_{\mathcal{S}} :  {\rm Perm}(S, {\mathcal{S}})  \to {\rm Sym}(d), \qquad \psi_{\mathcal{S}}(\tau) =  \prod_{i = 1}^k \tau|_i.
\end{align}
Note that in general $\psi_{\mathcal{S}}$ is not a group homomorphism, although it becomes one in the case where $\tau|_i$ commutes with $\omega|_j$ for any $\tau, \omega \in {\rm Perm}(S, {\mathcal{S}})$ and any $i, j$.

We are most interested in the following special case:
\begin{definition}
Let $G$ be a group and $S$ a $G$-set. A pair $({\mathcal{S}}, \sigma_{\mathcal{S}})$ is a {\em cyclic fiber system} for the action of $G$ on $S$ if ${\mathcal{S}}$ is a fiber system on $S$, the orbits of $\sigma_{\mathcal{S}}$ are precisely the sets in ${\mathcal{S}}$, and $G \leq C_{{\rm Sym}(S)}(\sigma_{\mathcal{S}})$, the centralizer in ${\rm Sym}(S)$ of $\sigma_{\mathcal{S}}$. We call $\sigma_{\mathcal{S}}$ the permutation associated to ${\mathcal{S}}$. 
\end{definition}

Let ${\mathcal{S}}$ be a cyclic fiber system for $G$, and for each $S_i$ in ${\mathcal{S}}$, fix an element $s_i$.  Suppose that $\tau \in C_{{\rm Sym}(S)}(\sigma_{\mathcal{S}})$, and $\tau(s_i) =  \sigma_{\mathcal{S}}^{r_i}(s_j)$. Because $\tau$ commutes with $\sigma_{\mathcal{S}}$, we have $\tau(\sigma_{\mathcal{S}}^t(s_i)) =  \sigma_{\mathcal{S}}^{r_i + t}(s_j)$ for all $t \geq 0$, and because $S_i$ is one of the orbits of $\sigma_{\mathcal{S}}$, this completely determines $\tau|_i$. Indeed, $\tau|_i = \delta^{r_i}$, where $\delta$ is the $d$-cycle $(0, 1, \ldots, d-1)$. The map in \eqref{wreathisom} becomes
\begin{align} \label{wreathisom2}
	\Phi : C_{{\rm Sym}(S)}(\sigma_{\mathcal{S}})  & \to ({\mathbb{Z}}/d{\mathbb{Z}}) \wr {\rm Sym}({\mathcal{S}}) \\
	 \tau \qquad  & \mapsto ((r_1, \ldots, r_k), \tau') \nonumber
\end{align}
We now obtain a homomorphism
\begin{align}  \label{psimap2}
	\psi_{\mathcal{S}} :  C_{{\rm Sym}(S)}(\sigma_{\mathcal{S}})  \to {\mathbb{Z}}/d{\mathbb{Z}}, \qquad \psi_{\mathcal{S}}(\tau) =  \sum_{i = 1}^k r_i.
\end{align}
Note that we have made a choice of the $s_i$, and $\Phi$ is not independent of this choice. Suppose that we replace $s_i$ with $s_i'$, and write $s_i' = \sigma_{\mathcal{S}}^\ell(s_i)$. One checks that $\tau|_i$ is now $r_i + \ell$. However, if $m$ is such that $\tau(S_m) = S_i$, then $\tau(s_m) = \sigma_{\mathcal{S}}^{r_m}(s_i) = \sigma_{\mathcal{S}}^{r_m-\ell}(s_i')$, and thus $\tau|_m = r_m - \ell$.  Hence the map $\psi_{\mathcal{S}}$ is independent of the choice of the $s_i$.

\subsection{Actions with multiple cyclic fiber systems}
Suppose $f(z) = z^d + c$ for some $d \geq 2$ and $K$ contains a primitive $d$th root of unity $\zeta_d$.
We describe two ways in which cyclic fiber systems arise for the action of the Galois group $G_n$ on the set $T_n$ of roots of $g \circ f^{n}$.  If $S = T_n$, then we obtain a fiber system by taking ${\mathcal{S}}$ to be the set of fibers of the map $f : T_n \to T_{n-1}$. We sometimes refer to this as the \textit{fundamental cyclic fiber system} of $T_n$. 
If $\alpha \in T_{n-1}$ and $\beta \in T_n$ satisfy $f(\beta) = \alpha$, then the fiber of the map $f$ over $\alpha$ is 
\begin{equation*} \label{preimages}
\{\beta \zeta_d^{j} : j = 0,1, \ldots, d-1\},
\end{equation*}
We make ${\mathcal{S}}$ into a cyclic fiber system by choosing $\sigma_{\mathcal{S}}$ to be the permutation given by multiplication by $\zeta_d$, which clearly acts as a full $d$-cycle on each fiber of $f$. Moreover, since $\zeta_d$ is fixed by each $\tau \in G_n$, we have that $\tau$ commutes with $\sigma_{\mathcal{S}}$, and therefore $({\mathcal{S}}, \sigma_{\mathcal{S}})$ is a cyclic fiber system for the action of $G_n$ on $T_n$. 

When $G_n$ has non-trivial center, we have another way to generate non-trivial cyclic fiber systems.  Take $\omega \in Z(G_n)$.  If $\{\omega^i(\beta) : i \geq 1\}$ is an orbit of $\omega$ acting on $T_n$ and $\tau \in G_n$, then $\tau(\{\omega^i(\beta) : i \geq 1\}) = \{\omega^i(\tau(\beta)) : i \geq 1\}$ and hence is another orbit of $\omega$. Thus if we denote the set of orbits of $\omega$ by ${\mathcal{O}}_\omega$, then the pair $({\mathcal{O}}_\omega, \omega)$ is a cyclic fiber system for $G_n$, which we call a {\em central cyclic fiber system}.  We remark that if $G_n$ acts transitively on $T_n$, then all orbits of $\omega$ must contain the same number of elements. 

A key difference between a central cyclic fiber system and the fundamental cyclic fiber system is that $\omega$ belongs to $G_n$, whereas a priori $\sigma_{\mathcal{S}}$ may not belong to $G_n$. In the case where the fundamental cyclic fiber system is also a central cyclic fiber system, we obtain $\sigma_{\mathcal{S}} \in G_n$, a conclusion that plays a crucial role in the proof of Theorem \ref{mart2}. We thus examine under what conditions a group action can have multiple distinct cyclic fiber systems. To fix ideas, and to give a flavor for our next result, we give an example. 
\begin{example} \label{a4}
Let $K = {\mathbb{Q}}$, $g(z) = z$, $f(z) =z^2 + 1/3$, and $T_2 = \{\pm \beta_1, \pm \beta_2\}$. One checks that both $f$ and $f^2$ are irreducible, and hence $\#G_2 \geq \deg f^2 = 4$.  However, the discriminant of $f^2$ is $1024/81$, which is a square, and thus 
$G_2 \leq A_4 \cap D_4$.  Hence $G_2 \cong A_4 \cap D_4$, and the action of $G_2$ on $T_2$ is given by 
$$e, \; (\beta_1, -\beta_1)(\beta_2, -\beta_2), \; (\beta_1, \beta_2)(-\beta_1, -\beta_2), \; (\beta_1, -\beta_2)(-\beta_1, \beta_2).$$
The fundamental cyclic fiber system for $G_2$ is $\{\beta_1,-\beta_1\}, \{\beta_2, -\beta_2\}$.  However, $G_2$ is abelian, and hence there are three non-trivial central cyclic fiber systems: the fundamental cyclic fiber system as well as the partitions $\{\{\beta_1, \beta_2\}, \{-\beta_1,-\beta_2\}\}$ and $\{\{\beta_1,-\beta_2\}, \{-\beta_1,\beta_2\}\}$. Note that $f^2(0) = 4/9$ is a square in ${\mathbb{Q}}$. 
\end{example}

The following is a generalization of \cite[Theorem 4.7]{galmart}.

\begin{lemma} \label{martlem1}
Let $G$ be a group acting transitively on a set $S$, and suppose that $({\mathcal{S}}, \sigma_{\mathcal{S}})$ is a cyclic fiber system for this action, with ${\mathcal{S}}$ composed of sets with $d$ elements. Let $({\mathcal{T}}, \sigma_{\mathcal{T}})$ be another cyclic fiber system for the action of $G$ on $S$, and suppose that $\sigma_{\mathcal{T}}$ commutes with $\sigma_{\mathcal{S}}$, $\sigma_{\mathcal{T}} \not\in \langle \sigma_{\mathcal{S}} \rangle$, and $\sigma_{\mathcal{T}}^d = 1$. Then $\psi_{\mathcal{S}}(G)$ is a proper subgroup of ${\mathbb{Z}}/d{\mathbb{Z}}$, where $\psi_{\mathcal{S}}$ is the restriction-product homomorphism given in \eqref{psimap2}.
\end{lemma}

\begin{proof}
By hypothesis the subgroup $H = \langle \sigma_{\mathcal{S}}, \sigma_{\mathcal{T}} \rangle$ of ${\rm Sym}(S)$ is abelian. Moreover, $G \leq C_{{\rm Sym}(S)}(\sigma_{\mathcal{S}}) \cap C_{{\rm Sym}(S)}(\sigma_{\mathcal{T}})$, and it follows that the orbits of $H$ form yet another fiber system for the action of $G$ on $S$. Put 
$$r = \min\{i \geq 1 : \sigma_{\mathcal{T}}^i \in \langle \sigma_{\mathcal{S}} \rangle \}.$$
Note that $\sigma_{\mathcal{T}}^i \in \langle \sigma_{\mathcal{S}} \rangle$ implies $\sigma_{\mathcal{T}}^{\gcd{(i,d)}} \in \langle \sigma_{\mathcal{S}} \rangle$, since $\sigma_{\mathcal{T}}^d = 1$. Therefore $r \mid d$, and moreover $r > 1$ by hypothesis. Note also that $|H| = rd$. 

We claim that $\psi_{\mathcal{S}}(G) \in \langle r \rangle \leq {\mathbb{Z}}/d{\mathbb{Z}}$. Let $B$ be a set of distinguished elements, one for each orbit of $\sigma_{\mathcal{S}}$. Now $H$ acts on $S$, and each orbit of this action consists of a disjoint union of $r$ orbits of $\sigma_{\mathcal{S}}$, which may be written as follows:
$$\begin{array}{cccc}
\beta_i & \sigma_{\mathcal{S}}(\beta_i) &  \ldots & \sigma_{\mathcal{S}}^{d-1}(\beta_i) \\
\sigma_{\mathcal{T}}(\beta_i) & \sigma_{\mathcal{T}}(\sigma_{\mathcal{S}}(\beta_i)) & \ldots & \sigma_{\mathcal{T}}(\sigma_{\mathcal{S}}^{d-1}(\beta_i)) \\
\vdots & \vdots && \vdots \\
\sigma_{\mathcal{T}}^{r-1}(\beta_i) & \sigma_{\mathcal{T}}^{r-1}(\sigma_{\mathcal{S}}(\beta_i)) & \ldots & \sigma_{\mathcal{T}}^{r-1}(\sigma_{\mathcal{S}}^{d-1}(\beta_i))
\end{array} 
$$
where $\beta_i, \sigma_{\mathcal{T}}(\beta_i), \ldots, \sigma_{\mathcal{T}}^{r-1}(\beta_i)$ may be assumed without loss of generality to lie in $B$. 
Let $g \in G$, and suppose that $g(\beta_i) = \sigma_{\mathcal{S}}^u \sigma_{\mathcal{T}}^v(\beta_j)$, where $0 \leq u \leq {d-1}$, $0 \leq v \leq r-1$, and $\beta_j \in B$. Then for each $s$ with $0 \leq s \leq r-1$, we have 
$$g(\sigma_{\mathcal{T}}^s(\beta_i)) = \sigma_{\mathcal{T}}^s(g(\beta_i)) = \sigma_{\mathcal{S}}^u \sigma_{\mathcal{T}}^{v+s}(\beta_j),$$
 and hence considering the restriction map with respect to ${\mathcal{S}}$ we obtain $g|_t = u$ for each of the $r$ choices of $t$ given by the elements of $\{\beta_i, \sigma_{\mathcal{T}}(\beta_i), \ldots, \sigma_{\mathcal{T}}^{r-1}(\beta_i)\}$. Since the same holds for every orbit of $H$, we get $\psi_{\mathcal{S}}(\tau) \in \langle r \rangle$. 
\end{proof}

\begin{lemma} \label{martlem2}
Let $K$ be a global field containing a primitive $d$th root of unity $\zeta_d$, let $f(z) = z^d + c \in K[z]$, and let $g(z) \in K[z]$ be monic.  Suppose that ${\mathcal{S}}$ is the fundamental cyclic fiber system for the action of $G_n$ on $T_n$, for some $n \geq 1$. If $\psi_{\mathcal{S}}(G_n)$ is a proper subgroup of ${\mathbb{Z}}/d{\mathbb{Z}}$, then $(-1)^{\epsilon} g(f^n(0))$ is an $r$th power in $K$ for some $r > 1$ with $r \mid d$, where $\epsilon = 1$ if $d$ is even, $n = 1$, and $\deg g$ is odd, and $\epsilon = 0$ otherwise.  
\end{lemma}

\begin{proof}
Suppose that $\psi_{\mathcal{S}}(G_n) = \langle r \rangle$, with $r \mid d$ and $r > 1$, and let $k := \deg(g \circ f^n)$ be the number of sets constituting the partition ${\mathcal{S}}$. Let $B = \{\beta_1, \ldots, \beta_k\}$ be a set of distinguished elements, one from each element of ${\mathcal{S}}$. Given $\tau \in G_n$ and $\beta_i \in B$, let $q_i$ be such that $\tau(\beta_i) = \beta_j\zeta_d^{q_i}$ for some $\beta_j \in B$. Because $\psi_{\mathcal{S}}(G_n) = \langle r \rangle$, we have that $\sum_{i=1}^k q_i = rs$ for some integer $s$.   Now
\[\tau \left( \prod_{\beta \in B} \beta \right)^{d/r} = \left( \prod_{\beta \in B} \beta \right)^{d/r} \cdot ((\zeta_d^{q_1 + \cdots + q_k})^{d/r})  =  \left( \prod_{\beta \in B} \beta \right)^{d/r}.\]
This holds for all $\tau \in G_n$, showing that $(\prod_{\beta \in B} \beta)^{d/r}$ is in the fixed field of $G_n$, and thus lies in $K$. Therefore $(\prod_{\beta \in B} \beta)^{d}$ is an $r$th power in $K$. On the other hand, the product of all roots of $g \circ f^{n}$ is
\begin{equation} \label{allroots}
	\prod_{\beta \in B} \prod_{i = 0}^{d-1} \zeta_d^i \beta = \prod_{\beta \in B} \zeta_d^{(d-1)d/2} \left(\prod_{\beta \in B} \beta \right)^{d}.
\end{equation}
Now $\prod_{\beta \in B} \zeta_d^{(d-1)d/2}$ is $-1$ if $d$ is even and $\#B$ is odd, and $1$ otherwise. But $\#B = \#{\mathcal{S}} = \deg (g \circ f^{n-1})$, and this is odd when $d$ is even only if $n = 1$ and $\deg g$ is odd. Hence the right-hand side of \eqref{allroots} is $(-1)^{\epsilon} (\prod_{\beta \in B} \beta)^{d}$. Finally, the product of all roots of $g \circ f^{n}$ is $(-1)^kg(f^n(0))$, where $k = \deg (g \circ f^{n})$. We thus obtain that $(-1)^{\epsilon + k}g(f^n(0))$ is an $r$th power in $K$.  If $d$ is odd, then this is an $r$th power in $K$ if and only if $g(f^n(0))$ is an $r$th power in $K$. If $d$ is even, then $(-1)^k = 1$. 
\end{proof}

\begin{lemma}\label{rafeisahiphopdog}
Let $K$ be a global field containing a primitive $d$th root of unity $\zeta_d$, let $f(z) = z^d + c \in K[z]$, and let $g(z) \in K[z]$ be monic with $g \circ f^{n}$ irreducible for some $n \geq 1$. Let ${\mathcal{S}}$ be the fundamental cyclic fiber system for the action of $G_n$ on $T_n$, $\sigma_{\mathcal{S}}$ the associated permutation, and $\epsilon$ as in Lemma \ref{martlem2}. If $(-1)^{\epsilon} g(f^n(0))$ is not an $r$th power in $K$ for any $r \mid d$ with $r > 1$, and the center of $G_n$ has an element of order $i$ with $i \mid d$, then $\sigma_{\mathcal{S}}^{d/i} \in G_n$. 
\end{lemma}

\begin{proof}
Let $\omega \in Z(G_n)$ have order $i$ with $i \mid d$, and let $({\mathcal{O}}_\omega, \omega)$ be the corresponding central cyclic fiber system. Because $i \mid d$, we have $\omega^d = 1$. Because $\omega \in G_n$ and $G_n \leq C_{{\rm Sym}(S)}(\sigma_{\mathcal{S}})$, we have that $\omega$ and $\sigma_{\mathcal{S}}$ commute. If $\omega \not\in \langle \sigma_{\mathcal{S}} \rangle$, then by Lemma \ref{martlem1} we have that $\psi_{\mathcal{S}}(G)$ is a proper subgroup of ${\mathbb{Z}}/d{\mathbb{Z}}$, which is impossible by Lemma \ref{martlem2}. Hence $\omega = \sigma_{\mathcal{S}}^j$ for some $j$. Because $|\omega| = i$ and $|\sigma_{\mathcal{S}}| = d$, there is a power of $\omega$ that gives 
$\sigma_{\mathcal{S}}^{d/i}$, and the lemma is proven.
\end{proof}

\begin{proof}[Proof of Theorem \ref{mart2}] 
We are assuming that $d$ is prime and $(-1)^{\epsilon}g(f^n(0))$ is not a $d$th power in $K$, and hence from Theorem \ref{firststab} we have that $g \circ f^{n}$ is irreducible over $K$ for all $n \geq 1$.  Moreover, the hypotheses that $d$ is prime and $g$ divides an iterate of $f$ imply that $K_n$ is formed from $K$ by repeatedly taking extensions of degree $d$, and hence for each $n \geq 1$, $G_n$ is a $d$-group.  Therefore the center of $G_n$ is non-trivial, and thus it must contain an element of order $d$. By Lemma \ref{rafeisahiphopdog}, we then have $\sigma_{{\mathcal{S}}} \in G_n$.  But $\sigma_{{\mathcal{S}}}$ fixes $K_{n-1}$ and acts on the roots of $f(z) - \alpha$ as a $d$-cycle, where $\alpha$ is any root of $g \circ f^{n-1}$. Hence $f(z)-\alpha$ is irreducible over $K_{n-1}$. This conclusion holds for all $n \geq 1$, and thus the theorem follows from Theorem 2.5 in \cite{quaddiv}.
\end{proof}

\section{Maximality Results and Proof of the Main Theorem} \label{max}

In this section we generalize a result of Stoll to give a criterion ensuring that the kernel of the projection $G_n \to G_{n-1}$ is as large as possible.  We then apply Siegel's theorem on integral points to certain curves to derive Theorem \ref{main1}.  Let notation and assumptions be as in Section \ref{GP(f,g)}. For $n \geq 1$, $K_n$ is obtained from $K_{n-1}$ by adjoining the $d$th roots of $m$ elements of $K_{n-1}$, where $m=\deg(g \circ f^{n-1})$. Setting $H_n={\rm Gal\,}(K_n/K_{n-1})$, we thus have an injection 
\[H_n \hookrightarrow ({\mathbb{Z}}/d{\mathbb{Z}})^m.\]
We call $H_n$ \textit{maximal} if this map is an isomorphism. 

\begin{lemma} \label{maxone}
Let $d \geq 2$ be an integer and let $K$ be a field of characteristic not dividing $d$ and containing a primitive $d$th root of unity. Let $f(z) = z^d + c \in K[z]$, and let $g(z) \in K[z]$ divide an iterate of $f$. Suppose that $n \geq 2$ and $g \circ f^{n-1}$ is irreducible over $K$. Then $H_n$ is maximal if and only if $g(f^n(0))$ is not a $p$-th power in $K(g \circ f^{n-1})$ for any prime $p \mid d$.
\end{lemma}

\begin{remark}
The lemma is false if $K$ does not contain a primitive $d$th root of unity. For instance, let $K = {\mathbb{Q}}$, $f(z) = z^3 + 1$, and $g(z) = z^2 - z + 1$, which divides $f(z)$. Then $g(f(z)) = z^6 + z^3 + 1$ is the $9$th cyclotomic polynomial, and hence is irreducible over ${\mathbb{Q}}$. Thus $G_1$ has order $6$ while a computer algebra system verifies that $G_2$ has order $2 \cdot 3^5$, whence $H_2$ has order $3^4$. However, $K_1 = {\mathbb{Q}}(\zeta_9)$, and one checks that $g(f^2(0)) = 3$ is not a cube in $K_1$.
\end{remark}

\begin{proof} This is an adaptation of Lemma 3.2 of \cite{quaddiv}, and thus is a generalization of Lemma 1.6 of \cite{stoll}.  Let $m = \deg(g \circ f^{n-1})$, and denote the roots of $g \circ f^{n-1}$ by $\beta_i$ for $i = 1, \dots, m$.  Note that $K_n$ is obtained by adjoining to $K_{n-1}$ the $d$-th roots of $\beta_i - c$ for $i = 1, \dots, m$, and hence $K_n/K_{n-1}$ is a $d$-Kummer extension. Moreover, since $n \geq 2$, $K_{n-1}$ contains $K_1$, and hence contains a primitive $d$th root of unity. Thus $[K_n : K_{n-1}] \leq d^m$. By \cite[Theorem 8.1, p. 295]{langalg}, $[K_n:K_{n-1}] = (B : K_{n-1}^{*d})$, where $B$ is the multiplicative subgroup generated by $\{\sqrt[d]{\beta_i - c} : i = 1, \ldots, m\}$ together with $K_{n-1}^{*d}$. It follows that $[K_n : K_{n-1}] < d^m$ if and only if there is a non-zero $(\epsilon_1, \dots, \epsilon_m) \in ({\mathbb{Z}}/d{\mathbb{Z}})^m$ such that $\prod^m_{i=1} (\beta_i - c)^{\epsilon_i}$ is a $d$-th power in $K_{n-1}$.  

By the irreducibility of $g \circ f^{n-1}$, we have that $G_n = {\rm Gal\,}(K_{n}/K)$ acts transitively on the $\beta_i$.  We then let $M$ be the $({\mathbb{Z}}/d{\mathbb{Z}})[G_n]$-module of all $(\epsilon_1, \dots, \epsilon_m) \in ({\mathbb{Z}}/d{\mathbb{Z}})^m$ such that $\prod^m_{i=1} (\beta_i - c)^{\epsilon_i}$ is a $d$-th power in $K_{n-1}$, where $G_n$ acts by permuting coordinates according to the action on the $\beta_i$.  From Lemma \ref{mginv} we have that $M \neq 0$ if and only if $M$ contains a $G_n$-invariant element.  By the transitivity of the action of $G_n$ on the $\beta_i$, such an element must have the form $(w, \dots, w)$ for some non-zero $w \in {\mathbb{Z}}/d{\mathbb{Z}}$.  Therefore $H_n $ is maximal if and only if $\prod^m_{i=1} (\beta_i - c) = (-1)^m g(f^{n-1}(c)) = (-1)^m g(f^n(0))$ is not an $r$-th power in $K_{n-1}$ for any $r \mid d$ (we can take $r = d/w'$ in the previous paragraph, where $w'$ is a divisor of $d$ generating $\langle w \rangle \leq {\mathbb{Z}} / d{\mathbb{Z}}$). Note that $m$ and $d$ must have the same parity, because we assume $n \geq 2$, so $(-1)^m$ is necessarily a $d$-th power. This proves the lemma. \end{proof}

Note also that $G_n$ is solvable, for it is a subgroup of the Galois group $B_{n+j}$ of $f^{n+j}(x)$ over $K$. If we let $N_i$ be the kernel of the restriction homomorphism $B_{n+j} \to B_{n+j-i}$, then clearly the $N_i$ form an ascending chain of normal subgroups of $B_{n+j}$, and moreover $N_{i}/N_{i-1}$ is isomorphic to the kernel of the restriction map $B_i \to B_{i-1}$, which is of the form $({\mathbb{Z}}/d{\mathbb{Z}})^{k_i}$. 

\begin{lemma} \label{mginv}
Let $G$ be a non-trivial solvable group whose order divides a power of $d$, and let $M \neq 0$ be a $({\mathbb{Z}}/d{\mathbb{Z}})[G]$-module.  Then the submodule $M^G$ of $G$-invariant elements is non-trivial.
\end{lemma}

\begin{proof}
We induct on the length of the composition series 
\[G = G_0 > G_1 > \cdots > G_k = \{e\}\]
such that each of the quotients $G_i/G_{i-1}$ are cyclic of prime order dividing $d$. First, suppose $G$ is cyclic of prime order dividing $d$, and take $0 \neq y \in M$.  Let $\sigma$ generate $G$; if $\sigma y = y$, we are done.  Otherwise, define $y_j$ for $j = 1, \dots, d-1$ to be
\[y_j = \sum_{\ell = 0}^{d-j} {d-\ell-1 \choose d - \ell - j} \sigma^{\ell} y\]
First, note that $y_1 = y + \sigma y + \dots + \sigma^{d-1} y$, so $\sigma y_1 = y_1$.  Then, since ${d-\ell \choose d - \ell - j +1} = {d-\ell-1 \choose d - \ell - j} + {d-\ell-1 \choose d - \ell - j +1}$, we have that $\sigma y_j = y_{j-1} + y_j$ for $1 < j \leq d-1$.  

If $y_1 \neq 0$, then we are done; if on the other hand $y_1 = 0$, then $\sigma y_2 = y_2$.  Similarly, if $y_j = 0$ for all $j < j'$, then $\sigma y_{j'} = y_{j'}$.  But note that 
\[y_{d-1} = {d-1 \choose 1} y + {d-2 \choose 0} \sigma y = -y+\sigma y.\]
This cannot be 0 by our initial assumption, so it cannot be the case that all the $y_j$'s are 0.  Therefore $M^G$ is non-trivial if $M$ is non-trivial.

If $G$ is not cyclic of prime order, then let $N$ be a non-trivial, proper maximal normal subgroup of  $G$, and note that both $N$ and $G/N$ are solvable with order dividing a power of $d$, and the length of the composition series of $N$ is strictly less than the length of $G$. Then $M$ is also a $({\mathbb{Z}}/d{\mathbb{Z}})[N]$-module, and by the induction hypothesis we have $M^N \neq 0$.  But now $M^N$ is a non-trivial $({\mathbb{Z}}/d{\mathbb{Z}})[G/N]$-module, so $(M^N)^{G/N} = M^G \neq 0$, again by the inductive hypothesis.
\end{proof}

Although we don't use it in our main argument, it may be of interest to have a criterion in terms of the ground field $K$ that ensures the maximality of $H_n$. The proof is essentially identical to the proof of Theorem 3.3 of \cite{quaddiv}, and follows from Lemma 2.6 of \cite{quaddiv} and Lemma \ref{maxone}.

\begin{theorem}
Let $d \geq 2$ be an integer, $K$ a global field of characteristic not dividing $d$, $f(z) = z^d + c \in K[z]$, and $g(z) \in K[z]$ divide an iterate of $f$.  Suppose that $n \geq 2$ and $g \circ f^{n-1}$ is irreducible over $K$, and denote by $v_{\mathfrak{p}}(g(f^n(0)))$ the valuation corresponding to the place ${\mathfrak{p}}$ of $K$. If there exists ${\mathfrak{p}}$ with $v_{\mathfrak{p}}(g(f^n(0)))$ prime to $d$, $v_{{\mathfrak{p}}}(g(f^{i}(0))) = 0$ for all $1 \leq i \leq n-1$, and $v_{{\mathfrak{p}}}(d) = 0$, then $H_n$ is maximal.
\end{theorem}

\begin{remark}
Assuming the ABC-conjecture of Masser-Oesterl\'{e}-Szpiro, it is shown in \cite[Theorem 1.4]{gratton} that if $K$ is a number field, $f(z) = z^d + c$, and $O_f(0)$ is infinite, then for all but finitely many $n$, there is a prime ${\mathfrak{p}}$ of $K$ with $v_{\mathfrak{p}}(g(f^n(0))) = 1$, $v_{{\mathfrak{p}}}(g(f^{i}(0))) = 0$ for $1 \leq i \leq n-1$, and $v_{{\mathfrak{p}}}(d) = 0$.  Hence $H_n$ is maximal for all but finitely many $n$, and it follows that $G_\infty$ has finite index in ${\rm Aut}(T)$. 
\end{remark}

\begin{theorem}\label{inf many primitives}
Let $d \geq 2$ be an integer, $K$ be a global field of characteristic not dividing $d$ and containing a $d$th root of unity, $f(z) = z^d + c \in K[z]$, and $g(z) \in K[z]$ divide an iterate of $f$.  Suppose that $g \circ f^{n}$ is irreducible for all $n \geq 1$ and that $O_f(0)$ is infinite. Then there are infinitely many $n$ such that $H_n$ is maximal. 
\end{theorem}

\begin{proof} 
Put $b_n = g(f^n(0))$ for $n \geq 1$, and let $j$ be such that $g(z) \mid f^j (z)$. Observe first that for any $n$, the coefficients of $f^{n}(z)$ are in ${\mathbb{Z}}[c]$. Hence if $v_{\mathfrak{p}}(c) \geq 0$ for some non-archimedean place of $K$, then the coefficients of $f^n(z)$ have nonnegative ${\mathfrak{p}}$-adic valuation, and thus the same holds for all its roots. Therefore $0 \leq v_{\mathfrak{p}}(b_n) \leq v_{\mathfrak{p}}(f^{n+j}(0))$.  

Let $\ell$ be a rational prime, and note that if ${\mathfrak{p}}$ is a non-archimedean place of $K$ with $v_{\mathfrak{p}}(b_{\ell - j}) > 0$ and $v_{\mathfrak{p}}(c) = 0$, then $v_{\mathfrak{p}}(b_i) = 0$ for $i = 1, \ldots, \ell-j-1$. Indeed, by the previous paragraph we have $v_{\mathfrak{p}}(f^{\ell}(0)) > 0$, and by Lemma \ref{rds1}, condition (2) of Definition \ref{RDS}, and the fact that $v_{\mathfrak{p}}(c) = 0$, this implies $v_{\mathfrak{p}}(f^{n}(0)) = 0$ for $n = 1, \ldots, \ell-1$. Since $0 \leq v_{\mathfrak{p}}(b_i) \leq v_{\mathfrak{p}}(f^{i + j}(0))$, we obtain the desired conclusion. 

We wish to work in a principal ideal domain. We create a set $S$ by selecting a finite set of places of $K$, containing all archimedean places, and adding to it the finitely many places at which $c$ has non-zero valuation. Then  the set ${\mathcal{O}}_{K,S_0}$ of $S$-integers is a principal ideal domain, and $b_n \in {\mathcal{O}}_{K,S}$ for each $n \geq 1$.

Now fix $r \in {\mathbb{Z}}$, $r > 1$, and denote by $U_{K,S}$ the set of $S$-units in $K$. Suppose that for infinitely many primes $\ell$, we have 
\begin{equation*}
	b_{\ell - j}  = uy^r,
\end{equation*}
for some $u \in U_{K,S}$ and $y \in {\mathcal{O}}_{K,S}$.  By absorbing $r$th powers into $y^r$, we may assume that $u$ belongs to a set of coset representatives of $U_{K,S}^r$. By Dirichlet's theorem on $S$-units \cite[p. 174]{FT} this set of representatives is finite. Since $O_f(0)$ is infinite, the sequence $\{f^n(0) : n \geq 1\}$ cannot have repeated values. The pigeonhole principle then dictates that there is some $u$ such that the curve
\begin{equation} \label{thecurve}
	C : g(f^3(z)) = uy^r
\end{equation}
has infinitely many points in ${\mathcal{O}}_{K,S}$ (with $z = f^{\ell - j - 3}(0)$). Assume for a moment that this gives a contradiction. Then for all but finitely many $\ell$, writing
\begin{equation*}
	b_{\ell - j}  = u'\pi_1^{e_1} \cdots \pi_k^{e_k},
\end{equation*}
with the $\pi_i$ irreducible in ${\mathcal{O}}_{K,S}$ and $u' \in U_{K,S}$, we must have $r \nmid e_i$ for some $i$. Denote by $v_{\mathfrak{p}}$ the place of $K$ corresponding to the prime ideal $\pi_i {\mathcal{O}}_{K,S}$, and note that $v_{\mathfrak{p}}(c) = 0$ by our choice of $S$ and $v_{\mathfrak{p}}(b_{\ell - j}) = e_i$. 

Applying this argument with $r$ varying over the distinct prime divisors $q_1, \ldots, q_t$ of $d$, we obtain that for all but finitely many $\ell$, there exist places ${\mathfrak{p}}_1, \ldots {\mathfrak{p}}_t$ such that $v_{{\mathfrak{p}}_i}(c) = 0$ and $v_{{\mathfrak{p}}_i}(b_{\ell - j})$ is not a multiple of $q_i$. From \cite[Lemma 2.6]{quaddiv}, the fact that $v_{{\mathfrak{p}}_i}(f^n(0)) = 0$ for $n = 1, \ldots, \ell - 1$ implies that ${\mathfrak{p}}_i$ does not divide the discriminant of $f^{\ell - 1}(z)$, and hence ${\mathfrak{p}}_i$ is unramified in $K(f^{\ell - 1})$ and thus also in $K(g \circ f^{\ell- j - 1})$.  So if $\mathfrak{P}_i$ is any prime of $K(g \circ f^{\ell- j - 1})$ lying above ${\mathfrak{p}}_i$, then $v_{\mathfrak{P}_i}(b_{\ell - j})$ is not a multiple of $q_i$, proving that $b_{\ell - j}$ is not a $q_i$th power in $K(g \circ f^{\ell- j - 1})$. Lemma \ref{maxone} then finishes the proof. 

Let us return now to the matter of the curve in \eqref{thecurve}. Because the characteristic of $K$ does not divide $d$, $g \circ f^{3}$ is separable of degree $\geq d^3$, and one easily verifies that the curve in \eqref{thecurve} has (absolute) genus at least two. When $K$ is a number field, this contradicts Siegel's theorem on $S$-integral points \cite[Theorem D.9.1]{jhsdioph}. Indeed, in this case we could take $g(f^2(z)) = uy^r$ in \eqref{thecurve}, as this ensures positive genus even in the case $d = 2$ and $\deg g = 1$. 

When $K$ is a global function field (with field of constants ${\mathbb{F}_q}$), there is no statement as clean as that of Siegel's theorem, and indeed there cannot be, for if $C$ is defined over ${\mathbb{F}_q}$ and $P = (y, z) \in C({\mathcal{O}}_{K,S}) \setminus C({\mathbb{F}_q})$, then $\{\sigma^n(P) : n \geq 1\}$ furnishes an infinite set of points in $C({\mathcal{O}}_{K,S})$, where $\sigma$ is the $q$th power Frobenius map, acting on the coordinates of $P$. Fortunately every infinite set of points in $C({\mathcal{O}}_{K,S})$ (indeed in $C(K)$) arises in this manner. A theorem of Samuel \cite[p. iv]{samuel}, building on work of Manin and Grauert, gives the following: if $C(K)$ is infinite, then after possibly replacing $K$ by a finite extension, $C$ is birationally equivalent over $K$ to a curve $C'$ defined over ${\mathbb{F}_q}$. Moreover, there is a finite collection of points $\Delta \subset C'(K) \setminus C'({\mathbb{F}_q})$ such that every point of $C'(K) \setminus C'({\mathbb{F}_q})$ is of the form $\sigma^n(P)$ for some $P \in \Delta$. Hence any infinite collection of points in $C'(K)$ must contain two points $P_1, P_2$ with  $P_2 = \sigma^{s}(P_{1})$ for some $s \geq 1$.  Because the Frobenius map $\sigma$ commutes with any rational map, there must be two similar points in $C(K)$, which we denote again by $P_1, P_2$. 

By construction, the points generated above on \eqref{thecurve} have $z = f^{\ell - j - 3}(0)$ for $j \geq 4$, and this holds in particular for the $z$-coordinate $z(P_i)$ of $P_i$. If there is an absolute value on $K$ with $|z(P_1)| < 1$, then it follows from Lemma \ref{rds1} that we must have $|z(P_2)| = |z(P_1)|$, which contradicts $P_2 = \sigma^{s}(P_1)$. If there is no such absolute value, then $z(P_1)$ is in the field of constants of $K$, and the same must hold for $c$ (otherwise there is an absolute value with $|f^n(0)| < 1$ for each $n \geq 1$). Thus the entire orbit of $0$ is contained in a finite field, and hence $O_f(0)$ is finite, contrary to our hypothesis.
\end{proof}

\begin{lemma}\label{count $X_n=t$}
Suppose that $H_n$ is maximal. Let $t\in{\mathbb{N}}$. Then $$\mu (Y_n=t\mid Y_{n-1}=t,Y_{n-2}=t,\dots,Y_{n-k}=t)\leq\frac{1}{2}.$$
\end{lemma}

\begin{proof}
Let $d_{n-1}=\text{deg}(g \circ f^{n-1})$. Suppose $\mu(Y_{n-1}=t,\dots ,Y_{n-k}=t)=s/\# G_{n-1}$. For $n\geq 1$, $t$ is either a multiple of $d$, or $\mu\{Y_n=t\}=0$, so we may replace $t$ with $dt$ to ease notation in the calculations below. As always, we assume $d\geq 2$.

Because $H_n$ is maximal, there are 
\[\binom{dt}{t}(d-1)^{dt-t}d^{d_{n-1}-dt}\] 
automorphisms of $G_n$ that restrict to any particular automorphism of $G_{n-1}$ fixing $dt$ roots.

The conditional probability $\mu (Y_n=dt\mid Y_{n-1}=dt,\dots,Y_{n-k}=dt)$
\begin{eqnarray*}
	&=&\left(\frac{s\binom{dt}{t} (d-1)^{dt-t}d^{d_{n-1}-dt}}{\# G_{n}}\right)\left(\frac{\# G_{n-1}}{s}\right)\\
	&=&\binom{dt}{t}\left(\frac{d-1}{d}\right)^{dt}(d-1)^{-t}\\
	&=&\left(\frac{d-1}{d}\right)^{dt}\prod_{r=0}^{t-1}\frac{dt-r}{(t-r)(d-1)}
\end{eqnarray*}

For fixed $r< t$, let $R(d,t)=\frac{dt-r}{(t-r)(d-1)}$. Both $\frac{\partial(R)}{\partial d}$ and $\frac{\partial(R)}{\partial t}$ are negative, and $\left(\frac{d-1}{d}\right)^{dt}$ is also decreasing as $d,t$ increase. Thus $\binom{dt}{t}\left(\frac{d-1}{d}\right)^{dt}(d-1)^{-t}$ takes its maximum of 1/2 at the minimum values for $d,t$, that is $d=2$ and $t=1$. \end{proof}

\begin{lemma}\label{$H_n$ max, fixers to zero}
If $GP(f,g)$ is  an eventual martingale and $H_n$ is maximal for infinitely many $n$, then \[\lim_{n\to\infty}\mu(Y_n>0)=0.\]
\end{lemma}

\begin{proof}
As $GP(f,g)$ is an eventual martingale, it converges in probability by Doob's theorem. (See, e.g. \cite{Stoch}.)  Let $Y=\lim_{n\to\infty}Y_n$. Let $t\in{\mathbb{N}}$ and suppose that $\mu\{Y=t\}>0$.  There exists $m\in{\mathbb{N}}$ and $r\in{\mathbb{Q}}_{>0}$ such that 
\[\mu\left(\cap_{i\geq m}\{Y_i=t\}\right)=r>0,\]
because the $Y_n$ are integer-valued.  We fix $t\in{\mathbb{N}}$. Let $\mathcal{C}_i=\{Y_i=t\}$.
\[r\leq\mu\left(\cap_{i\geq m}\mathcal{C}_i\right)\leq \mu\left(\cap_{i=m}^k\mathcal{C}_i\right)\]
for any integer $k>m$.
\begin{equation*}\label{conditional}
	\mu(\cap_{i=m}^{k}\mathcal{C}_i)=\mu(\mathcal{C}_k\mid \cap_{i=m}^{k-1} \mathcal{C}_i)\cdot\mu(\mathcal{C}_{k-1}\mid \cap_{i=m}^{k-2} \mathcal{C}_i)\dots\mu(\mathcal{C}_m)
\end{equation*}
Suppose that $H_n$ is maximal for $s$ values of $n$ between $m$ and $k$. Then $r<\mu(\cap_{i=m}^{k}\mathcal{C}_i)\leq\frac{1}{2}^s$, from Lemma \ref{count $X_n=t$}. We let $k$ go to infinity, and, since $H_n$ is maximal for infinitely many $n$, $s$ goes to infinity as well. Then  
\[r<\lim_{s\to\infty}\left(\frac{1}{2}\right)^s.\]
This conclusion is false, therefore $\mu\{Y=t\}=0$ for all $t>0$. \end{proof}

We are at last in position to prove our main result. 

\begin{proof}[Proof of Theorem \ref{main1}]
In both case (1) and case (2), we find that $f$ is eventually stable. That is, there is $j \in {\mathbb{Z}}_{\geq 1}$ such that $f^j(z)=\prod_{i=1}^t g_i(z)$, with $g_i(f^n(z))$ irreducible for all $n \geq 0$. In case (1) this follows from Theorem \ref{evstab}, while in case (2) it follows from Theorem \ref{firststab}. Recall that $P_{f,g_i}(a_0)$ is the set of prime ideals ${\mathfrak{q}}$ of ${\mathcal{O}}_K$ such that ${\mathfrak{q}} \mid g_i(f^n(a_0))$ for at least one $n \geq 1$. Clearly ${\mathfrak{q}}\in P_{f,g_i}(a_0)$ for some $1\leq i\leq t$ if and only if  ${\mathfrak{q}}\in P_{f}(a_0)$.  We observe now that the Galois process associated to $(f, g_i)$ is an eventual martingale; in case (1) this is a consequence of Theorem \ref{mart1}, while in case (2) it follows from Theorem \ref{mart2}. Because $g_i(f^n(z))$ is irreducible for all $n \geq 0$ and $O_f(0)$ is infinite by hypothesis, we may apply Theorem \ref{inf many primitives} to conclude that $H_n$ is maximal for infinitely many $n$. From Lemma \ref{$H_n$ max, fixers to zero} and Theorem \ref{densecon}, we then have that $D(P_{f,g_i}(a_0))=0$. Therefore $P_{f}(a_0)$ is a finite union of zero-density sets, proving the theorem.  
\end{proof}

\section{The Case of $z^p + c \in {\mathbb{Z}}[z]$}  \label{zcasesec}

In this section we prove Corollary \ref{maincor2} by studying the family $f(z) = z^p + c$, $c \in {\mathbb{Z}} \setminus \{0\}$ over the field ${\mathbb{Q}}(\zeta_p)$. In particular, we may apply part (2) of Theorem \ref{main1} to members of this family (excepting the case $p = 2$ and $c = -1$), with $j = 1$ when $c$ is a $p$th power in ${\mathbb{Z}}$, and $j = 0$ otherwise. This follows from Lemma \ref{zcase} and the remark immediately after it.

\begin{lemma} \label{zcase}
Let $f(z) = z^p + c$, $c \in {\mathbb{Z}} \setminus \{0\}$, and let $p$ be an odd prime. If $c$ is not a $p$th power in ${\mathbb{Z}}$, then $f^n(0)$ is not a $p$th power in ${\mathbb{Z}}$ (and hence in ${\mathbb{Q}}(\zeta_p)$) for all $n \geq 1$. If $c = r^p$ for some $r \in {\mathbb{Z}}$, then no element of the form
\begin{equation} \label{factor}
	(f^{n-1}(0) + r \zeta_p^i),  \qquad i \geq 0
\end{equation}
is a $p$th power in ${\mathbb{Q}}(\zeta_p)$, for any $n \geq 2$. 
\end{lemma}

\begin{remark}
The case $p = 2$ is handled in \cite[Proposition 4.5]{quaddiv}, which gives that $f^n(0)$ is not a $p$th power provided that $c \neq -r^2$. Moreover, if $c = -r^2$ for $r \neq \pm 1$, then no element of the form $f^{n-1}(0) \pm r$ is a square in ${\mathbb{Q}}$. 
\end{remark}

Before proving Lemma \ref{zcase}, we give two corollaries.

\begin{corollary} \label{zfactors}
Let $f(z) = z^p + c$, $c \in {\mathbb{Z}} \setminus \{0\}$, and let $p$ be an odd prime. Over ${\mathbb{Q}}(\zeta_p)$, $f(z)$ is stable if $c$ is not of the form $r^p$, $r \in {\mathbb{Z}}$. Otherwise, $f(z)$ is the product of the $p$ linear polynomials $g_i(z) = z + r\zeta_p^i$, $i = 0, \ldots, p-1$, and $g_i(f^n(z))$ is irreducible for all $n \geq 1$. Over ${\mathbb{Q}}$, $f(z)$ is stable if $c$ is not of the form $r^p$, $r \in {\mathbb{Z}}$. Otherwise, $f(z) = (z-r)h(z)$ for some irreducible $h(z) \in {\mathbb{Z}}[z]$, and $f^n(z) - r$ and $h(f^n(z))$ are both irreducible for all $n \geq 1$.
\end{corollary}

\begin{proof}
The first assertion is an application of Lemma \ref{zcase} and Theorem \ref{firststab}. For the second assertion, note that if $c = r^p$, then the roots of the $g_i(z)$ are Galois conjugate for $i = 1, \ldots, p-1$, and the same is true for the roots of $g_i(f^n(z))$. Hence letting $h(z) = \prod_{i = 1}^{p-1} g_i(z)$ gives the desired result. 
\end{proof}

\begin{corollary} \label{orbitbounds}
Let $f(z) = z^p + c$, $c \in {\mathbb{Z}} \setminus \{0\}$, and let $p$ be an odd prime. If $K = {\mathbb{Q}}(\zeta_p)$, then the action of ${\rm Gal\,}(\overline{K}/K)$ on the roots of $f^n(z)$ has at most $p$ orbits, for any $n \geq 1$. If $K = {\mathbb{Q}}$, the corresponding Galois action on the roots of any iterate has at most two orbits.
\end{corollary}

Corollary \ref{orbitbounds} proves the corresponding cases of Sookdeo's conjecture on integral points in backwards orbits \cite[Conjecture 1.2]{sookdeo}. \label{ingramdisc} It also provides an interesting counterpart to a result of Ingram \cite{ingram}, where it is shown that the number of orbits of the Galois action on roots of $f^n(z) - a$ remains bounded as $n$ grows, provided that there exists a prime ${\mathfrak{p}}$ of $K$ with ${\mathfrak{p}} \nmid n$ and $|f^n(a)|_{\mathfrak{p}} \to \infty$ as $n \to \infty$. Corollary \ref{orbitbounds} corresponds to the case $a = 0$, and since $c$ is an integer, $f^n(0)$ is an integer for each $n \geq 1$, implying that $|f^n(0)|_{\mathfrak{p}} \leq 1$ for all $n \geq 1$ and for all ${\mathfrak{p}}$. Hence Corollary \ref{orbitbounds} provides information beyond Ingram's result. Ingram's methods involve giving a Galois-equivariant ${\mathfrak{p}}$-adic power series that conjugates $f$ to $z^d$ on a neighborhood of infinity. 

\begin{proof}[Proof of Lemma \ref{zcase}]
We remark that if $y \in {\mathbb{Z}}$ is not a $p$th power in ${\mathbb{Z}}$, then it is not a $p$th power in ${\mathbb{Q}}(\zeta_p)$. Indeed, since $p$ is odd prime, $y$ must have a prime factor $q \in {\mathbb{Z}}$ occurring to a power not divisible by $p$. But $q$ is unramified if $q \neq p$ and otherwise $q = {\mathfrak{p}}^{p-1}$ for some prime ${\mathfrak{p}}$ of ${\mathbb{Q}}(\zeta_p)$. In either case, the prime ideal factorization of $(q)$ in ${\mathcal{O}}_{{\mathbb{Q}}(\zeta_p)}$ has a prime occurring to power not divisible by $p$. Hence $q$, and therefore $y$, is not a $p$th power in ${\mathbb{Q}}(\zeta_p)$.

It is enough to prove the lemma in the case $c > 0$, for if $f_c = z^p + c$ and $f_{-c} = z^p - c$, then $f_{-c}^n(0) = -f_c^n(0)$. Thus we suppose that $c > 0$. For any positive integer $y$, we have 
\begin{equation} \label{yplusone}
	(y+1)^p - y^p = \sum_{i = 0}^{p-1} {p\choose{i}} y^i > py^{p-1}.
\end{equation}
and therefore $y^p + c$ is not a $p$th power when $0 < c < py^{p-1}$. If $y \geq c$, then clearly this holds. But $f^{n-1}(0) \geq c$ for all $n \geq 2$, and hence $(f^{n-1}(0))^p + c = f^n(0)$ is not a $p$th power in ${\mathbb{Z}}$ for all $n \geq 2$. Therefore if $c$ is not a $p$th power in ${\mathbb{Z}}$, then $f^n(0)$ is not a $p$th power in ${\mathbb{Z}}$ for all $n \geq 1$.

Suppose that $c = r^p$ for some positive integer $r$.  We handle first the case $i = 0$, where \eqref{factor} takes the values $r^p + r, r^{p^2} + r^p + r, (r^{p^2} + r^p)^p + r^p + r, É$ for $n = 2, 3, 4, \ldots$. As above, we have that $y^p + (r^p + r)$ is not a $p$th power in ${\mathbb{Z}}$ provided $0 < r^p + r < py^{p-1}$. This clearly holds if $y \geq r^p$. But $f^{n-2}(0) \geq r^p$ for all $n \geq 3$, and hence $f^{n-1}(0) + r = (f^{n-2}(0))^p + r^p + r$ is not a $p$th power for all $n \geq 3$. Observe also that $f(0) + r = r^p + r$ lies strictly between $r^p$ and $(r+1)^p$, and thus is not a $p$th power.

Suppose now that  $0 < i < p$. Because $p$ is prime, elements of the form $f^{n-1}(0) + r \zeta_p^i$ with $0 < i < p$ are Galois conjugate, and hence have identical norms. Thus it is enough to show that $N_{{\mathbb{Q}}(\zeta_p)/{\mathbb{Q}}}(f^{n-1}(0) + r \zeta_p) := N(f^{n-1}(0) + r \zeta_p)$ is not a $p$th power in ${\mathbb{Z}}$.  First, note that 
\begin{align}
	N(t + r \zeta_p) & = \sum_{i = 0}^{p-1} (-1)^{i}r^it^{p-1-i}  = t^{p-1} - rt^{p-2} + \cdots - r^{p-2}t + r^{p-1} \nonumber \\
	& = t^{p-1} - rt^{p-3}(t - r) - \cdots - r^{p-2}(t - r) \label{normeq1} \\
	& = t^{p-1} - rt^{p-2} + r^2t^{p-4}(t - r) + \cdots + r^{p-3}t(t - r) + r^{p-1}. \label{normeq2}
\end{align}
Then if $y \geq r^p$ and $t = y^p + r^p$, we have $t > r^p \geq r$, and thus from \eqref{normeq1} we obtain
\begin{align*}
	N(t + r \zeta_p) & < t^{p-1} = (y^p + r^p)^{p-1} \leq (y^p + y)^{p-1} \\
	& = y^{p-1}(y^{p-1} + 1)^{p-1} < (y^{p-1} + 1)^p.
\end{align*}
On the other hand, from \eqref{normeq2}, we have
\begin{align*}
	N(t + r \zeta_p) & > t^{p-1} - rt^{p-2} = t^{p-2}(t - r) > (t-r)^{p-1} \\
	& = (y^p + r^p-r)^{p-1} \geq (y^p)^{p-1} = (y^{p-1})^p.
\end{align*}
Now take $t = f^{n-1}(0) = (f^{n-2}(0))^p + r^p$; so when $n \geq 3$ we have $y = f^{n-2} (0) \geq r^p$. Therefore $N(f^{n-1}(0) + r \zeta_p)$ is not a $p$th power in ${\mathbb{Z}}$ for $n \geq 3$. 

If $n = 2$, then $y = 0$ and $t = r^p$ in the above calculation. When $r > 1$, we have $t > r$, and hence from \eqref{normeq1} and \eqref{normeq2} we obtain 
\begin{equation} \label{rp}
	(r^p)^{p-1} - r(r^p)^{p-2} < N(r^p + r \zeta_p) < (r^p)^{p-1},
\end{equation} 
Note that if $x \geq 1$, then $(x-1)^k \leq x^k - x^{k-1}$, as can be seen by multiplying both sides of the obvious inequality $(x-1)^{k-1} \leq x^{k-1}$ by $(x-1)$. Thus
\[(r^{p-1} - 1)^p \leq (r^{p-1})^p - (r^{p-1})^{p-1} = (r^p)^{p-1} - r(r^p)^{p-2}.\]
From \eqref{rp} we now have that $N(r^p + r \zeta_p)$ is not a $p$th power in ${\mathbb{Z}}$ when $r > 1$. 

When $n = 2$ and $r = 1$, we must adopt a different approach, since $N(1 + \zeta_p) = 1$.  We show that $1+ \zeta_p$ is not a $p$-th power separately in Lemma \ref{heightlem}, completing the proof of the present lemma. 
\end{proof}

\begin{lemma} \label{heightlem}
Let $p$ be an odd prime. Then $1 + \zeta_p$ is not a $p$th power in ${\mathbb{Q}}(\zeta_p)$. 
\end{lemma}

\begin{proof}
First, note that it suffices to show that $1 + \zeta_p$ is not a $p$th power in ${\mathbb{Z}}[\zeta_p]$.  Indeed, $1 + \zeta_p$ is a unit in ${\mathbb{Q}}(\zeta_p)$, and so any $p$-th root of $1 + \zeta_p$ must also be a unit; but the units of ${\mathbb{Q}}(\zeta_p)$ are contained in ${\mathbb{Z}}[\zeta_p]$.

Then, suppose there exists $x \in {\mathbb{Z}}[\zeta_p]$ such that $x^p = 1 + \zeta_p$.  If we reduce this equation mod $p$, the left hand side must be congruent to an integer; that is, there exists $n \in {\mathbb{Z}}$ and $y \in {\mathbb{Z}}[\zeta_p]$ such that $x^p = n + py = 1 + \zeta_p$, which is clearly impossible.  
\end{proof}


\begin{thebibliography}{}

\bibitem{ballot}
Ballot, C. (1995).
\newblock Density of prime divisors of linear recurrences.
\newblock {\em Mem. Amer. Math. Soc.\/}~{\em 115\/}(551), viii+102.

\bibitem{periodsmoduloprimes}
Benedetto, R.~L., D.~Ghioca, B.~Hutz, P.~Kurlberg, T.~Scanlon, and T.~J.
  Tucker.
\newblock Periods of rational maps modulo primes.
\newblock To appear in \textit{Math. Ann.}

\bibitem{ML}
Benedetto, R.~L., D.~Ghioca, P.~Kurlberg, and T.~J. Tucker (2012).
\newblock A case of the dynamical {M}ordell-{L}ang conjecture.
\newblock {\em Math. Ann.\/}~{\em 352\/}(1), 1--26.

\bibitem{settled}
Boston, N. and R.~Jones (2012).
\newblock Settled polynomials over finite fields.
\newblock {\em Proc. Amer. Math. Soc.\/}~{\em 140\/}(6), 1849--1863.

\bibitem{Stoch}
Brze{\'z}niak, Z. and T.~Zastawniak (1999).
\newblock {\em Basic stochastic processes}.
\newblock Springer Undergraduate Mathematics Series. London: Springer-Verlag
  London Ltd.
\newblock A course through exercises.

\bibitem{Cornell}
{Cornell}, G. (1982).
\newblock {On the construction of relative genus fields}.
\newblock {\em Trans. Amer. Math. Soc.\/}~{\em 271\/}(2), 501--511.

\bibitem{danielson}
Danielson, L. and B.~Fein (2002).
\newblock On the irreducibility of the iterates of {$x^n-b$}.
\newblock {\em Proc. Amer. Math. Soc.\/}~{\em 130\/}(6), 1589--1596
  (electronic).

\bibitem{doerksen}
Doerksen, K. and A.~Haensch (2012).
\newblock Primitive prime divisors in zero orbits of polynomials.
\newblock {\em Integers\/}~{\em 12}, A9, 7.

\bibitem{recseq}
Everest, G., A.~van~der Poorten, I.~Shparlinski, and T.~Ward (2003).
\newblock {\em Recurrence sequences}, Volume 104 of {\em Mathematical Surveys
  and Monographs}.
\newblock Providence, RI: American Mathematical Society.

\bibitem{faber-hutz}
Faber, X., B.~Hutz, and M.~Stoll (2011).
\newblock On the number of rational iterated preimages of the origin under
  quadratic dynamical systems.
\newblock {\em Int. J. Number Theory\/}~{\em 7\/}(7), 1781--1806.

\bibitem{FT}
Fr{\"o}hlich, A. and M.~J. Taylor (1993).
\newblock {\em Algebraic number theory}, Volume~27 of {\em Cambridge Studies in
  Advanced Mathematics}.
\newblock Cambridge: Cambridge University Press.

\bibitem{ostafe}
Gomez-Perez, D., A.~Ostafe, and I.~E. Shparlinski.
\newblock {On irreducible divisors of iterated polynomials}.
\newblock Preprint, 2012.

\bibitem{gratton}
Gratton, C., K.~Nguyen, and T.~J. Tucker.
\newblock {ABC implies primitive prime divisors in arithmetic dynamic}.
\newblock To appear, \textit{Bull. Lond. Math. Soc.} Available at
  http://arxiv.org/abs/1208.2989.

\bibitem{jhsdioph}
Hindry, M. and J.~H. Silverman (2000).
\newblock {\em Diophantine geometry}, Volume 201 of {\em Graduate Texts in
  Mathematics}.
\newblock New York: Springer-Verlag.
\newblock An introduction.

\bibitem{ingram}
Ingram, P.
\newblock {Arboreal Galois representations and uniformization of polynomial
  dynamics}.
\newblock Available at http://arxiv.org/abs/1111.3607.

\bibitem{ingramcanon}
Ingram, P. (2009).
\newblock Lower bounds on the canonical height associated to the morphism
  {$\phi(z)=z^d+c$}.
\newblock {\em Monatsh. Math.\/}~{\em 157\/}(1), 69--89.

\bibitem{galmart}
Jones, R. (2007).
\newblock Iterated {G}alois towers, their associated martingales, and the
  {$p$}-adic {M}andelbrot set.
\newblock {\em Compos. Math.\/}~{\em 143\/}(5), 1108--1126.

\bibitem{quaddiv}
Jones, R. (2008).
\newblock The density of prime divisors in the arithmetic dynamics of quadratic
  polynomials.
\newblock {\em J. Lond. Math. Soc. (2)\/}~{\em 78\/}(2), 523--544.

\bibitem{itconst}
{Jones}, R. (2012).
\newblock {An iterative construction of irreducible polynomials reducible
  modulo every prime}.
\newblock {\em J. Algebra\/}~{\em 369}, 114--128.

\bibitem{krieger}
Krieger, H.
\newblock {Primitive prime divisors in the critical orbit of $z^d + c$}.
\newblock To appear, \textit{Int. Math. Res. Not}. Available at
  http://arxiv.org/abs/1203.2555v2.

\bibitem{langalg}
Lang, S. (2002).
\newblock {\em Algebra\/} (third ed.), Volume 211 of {\em Graduate Texts in
  Mathematics}.
\newblock New York: Springer-Verlag.

\bibitem{Lorenzini}
Lorenzini, D. (1996).
\newblock {\em An invitation to arithmetic geometry}, Volume~9 of {\em Graduate
  Studies in Mathematics}.
\newblock Providence, RI: American Mathematical Society.

\bibitem{neukirch}
Neukirch, J. (1999).
\newblock {\em Algebraic number theory}, Volume 322 of {\em Grundlehren der
  Mathematischen Wissenschaften [Fundamental Principles of Mathematical
  Sciences]}.
\newblock Berlin: Springer-Verlag.
\newblock Translated from the 1992 German original and with a note by Norbert
  Schappacher, With a foreword by G. Harder.

\bibitem{rice}
Rice, B. (2007).
\newblock Primitive prime divisors in polynomial arithmetic dynamics.
\newblock {\em Integers\/}~{\em 7}, A26, 16.

\bibitem{Rosen}
Rosen, M. (2002).
\newblock {\em Number theory in function fields}, Volume 210 of {\em Graduate
  Texts in Mathematics}.
\newblock New York: Springer-Verlag.

\bibitem{samuel}
Samuel, P. (1966).
\newblock {\em Lectures on old and new results on algebraic curves}.
\newblock Notes by S. Anantharaman. Tata Institute of Fundamental Research
  Lectures on Mathematics, No. 36. Bombay: Tata Institute of Fundamental
  Research.

\bibitem{sookdeo}
Sookdeo, V.~A. (2011).
\newblock Integer points in backward orbits.
\newblock {\em J. Number Theory\/}~{\em 131\/}(7), 1229--1239.

\bibitem{stoll}
Stoll, M. (1992).
\newblock Galois groups over {${\bf Q}$} of some iterated polynomials.
\newblock {\em Arch. Math. (Basel)\/}~{\em 59\/}(3), 239--244.

\end{thebibliography}
\end{document}